\documentclass[11pt,amsfonts]{article}
\usepackage[margin=1in]{geometry}  
\usepackage{graphicx}              
\usepackage{amsmath, amssymb, amsthm, epsfig}               
\usepackage{amsfonts}              
\usepackage{amsthm}                
\usepackage[pdftex]{hyperref}
\usepackage{color}
\def\nd{\noindent}
\def\thend{\rule{3mm}{3mm}}

\newtheorem{theorem}{Theorem}[section]
\newtheorem{prop}{Proposition}[section]
\newtheorem{lemma}{Lemma}[section]

\newtheorem{cor}{Corollary}[section]

\numberwithin{equation}{section}

\begin{document}
\title{A Berestycki-Lions type result for a class of degenerate elliptic problems involving the Grushin Operator}
\author{\sf Claudianor O. Alves \thanks{Claudianor Alves was partially supported by CNPq/Brazil Proc. 304804/2017-7 } \,\,\, and \,\,\, Angelo R. F. de Holanda }
\maketitle

\begin{abstract}
In this work we study the existence of nontrivial solution  for the following class of semilinear degenerate elliptic equations
$$
-\Delta_{\gamma} u + a(z)u = f(u) ~~ \mbox{in} ~~ \mathbb{R}^{N},
$$
where $\Delta_{\gamma}$ is known as the {\it Grushin operator}, $z:=(x,y)\in\mathbb{R}^{m}\times\mathbb{R}^{k}$ and $m+k=N\geq 3$, $f$ and $a$ are continuous function satisfying some technical conditions. In order to overcome some difficulties involving this type of operator, we have proved some compactness results that are crucial in the proof of our main results. For the case $a=1$, we have showed a Berestycki-Lions type result. 

\end{abstract}

{\scriptsize \textbf{2000 Mathematics Subject Classification:} 35A15,35J70 , 35R01}

{\scriptsize \textbf{Keywords:} Variational Methods, degenerate elliptic equations, PDE on manifolds}


\section{Introduction}
In this paper we study the existence of nontrivial solutions for the following class of semilinear degenerate elliptic equations
$$ 
-\Delta_{\gamma} u + a(z)u  =  f(u) \quad  \mbox{in} \quad \mathbb{R}^{N}, \leqno{(P)_{a}}
$$
where $\Delta_{\gamma}$ is given by
\begin{equation}
\Delta_{\gamma}u(z)=\Delta_{x}u(z)+|x|^{2\gamma}\Delta_{y}u(z)
\end{equation}
is known as the {\it Grushin operator}, $\Delta_x$ and $\Delta_y$ are the Laplace operators in the variables $x$ and $y$, respectively, with $z:=(x,y)\in\mathbb{R}^{m}\times\mathbb{R}^{k}$ and $m+k=N\geq 3$. Here, $\gamma\geq 0$ is a real number and $N_{\gamma}=m+(1+\gamma)k$ is the appropriate homogeneous dimension. The nonlinearity $f:\mathbb{R} \longrightarrow \mathbb{R}$ and the weight $a:\mathbb{R}^{N}\longrightarrow\mathbb{R}$ are continuous functions that satisfy some technical conditions. 

To begin with, we would like to observe that $\Delta_{\gamma}$ is elliptic for $x\in\mathbb{R}^{m}$, $|x|\neq 0$ and degenerates on the manifold $\{0\}\times\mathbb{R}^{k}$. This operator was initially studied by {\it Grushin} in \cite{Grushin1970} and \cite{Grushin1971}, and it appears in the study of P.D.E. on  manifolds. For example, in Kogoj and Lanconelli \cite[Page 4641]{Kogoj&Lanconelli2012NA} the authors showed that problem $(P)_a$ appears in the study of a P.D.E. in the Heinsenberg group, while in	Bhakta and Sandeep \cite{MS}, 
it was showed a strong relation with a P.D.E on the Hyperbolic space. Furthermore, if $\gamma=0$, the operator $\Delta_{0}$ is the well-known classical Laplace operator. On the other hand, $\Delta_{\gamma}$ is contained in a family of operators of the kind 
\begin{equation}\label{general_operator}
\Delta_{\lambda}:=\sum_{i=1}^{N}\partial_{x_i}(\lambda_i^2\partial_{x_i}), ~~ \lambda=(\lambda_1,\cdots,\lambda_N),
\end{equation}
which has been extensively studied. See, for instance Anh and My\cite{Anh&My2016CCVE}, Kogoj and Lanconelli \cite{Kogoj&Lanconelli2012NA}, Luyen and tri \cite{Luyen&Tri2019CVEE}, Rahal and Hamdani  \cite{Rahal&Hamdani2018JFT} and references therein. 

At the last years some authors have dedicated a special attention for problems involving the {\it Grushin operator}, which we quote, Chung \cite{Chung2014CKMS}, D'Ambrosio \cite{DAmbrosio2003PAMS}, Monti \cite{Monti2006CPDE}, Monti and Morbidelli \cite{Monti&Morbidelli2006DMJ}, and more recently Duong and Nguyen \cite{Duong&Nguyen2017EJDE}, Duong and Phan \cite{Duong&Phan2017JMAA}, Liu, Tang and Wang \cite{Liu&Tang&Wang2019AMPA}, Loiudice \cite{Loiudice2019AMPA}.

The main motivation to study the above problem $(P)_a$ comes from of the papers due to Berestycki and Lions \cite{berest}, and Berestycki, Gallouet and Kavian \cite{BGK} that have considered the case $\gamma=0$ and $a(z)=1$ for all $z\in\mathbb{R}^N$. More precisely, in \cite{berest}, Berestycki and Lions have considered the existence of solution for 
\begin{equation}\label{BL}
- \Delta u   = g(u), \quad \mbox{in} \quad \mathbb{R}^N, 
\end{equation} 
by assuming that $N \geq 3$ and the following conditions on $g$:
$$
- \infty < \liminf_{s \to 0^+}\frac{g(s)}{s} \leq \limsup_{s \to 0^+}\frac{g(s)}{s}\leq -m<0, \leqno{(g_1)} 
$$
$$
\limsup_{s \to +\infty}\frac{g(s)}{s^{2^{*}-1}}\leq 0, \leqno{(g_2)}
$$
$$
\mbox{there is} \quad \xi>0 \, \, \mbox{such that} \,\, G(\xi)>0, \leqno{(g_3)}
$$
where $G(s)=\int_{0}^{s}g(t)\,dt$.

In \cite{BGK}, Berestycki, Gallouet and Kavian have studied the case where $N=2$ and the nonlinearity $g$ possesses an exponential growth of the type
$$
\limsup_{s \to +\infty}\frac{g(s)}{e^{\beta s^2}}=0, \quad \forall \beta >0.
$$

In the two  above mentioned papers, the authors obtained a solution for (\ref{BL}) by solving in the first moment a minimization problem, where the properties of the radially symmetric functions play an essential rule in their approach. After that, by using Lagrange multipliers, they were able to get a nontrivial solution for (\ref{BL}).

A version of the problem (\ref{BL}) for the critical case have been  made in Alves, Souto and Montenegro \cite{AlvesSoutoMontenegro} for $N \geq 3$ and $N=2$, see also Zhang and Zhou \cite{ZZ} for $N=3$. The reader can found in Alves, Figueiredo and Siciliano \cite{AGS}, Chang and Wang \cite{ChangWang} and Zhang, do \'O and Squassina \cite{ZOS} the same type of results involving the fractional Laplacian operator, more precisely, for a problem like 
\begin{equation}\label{fracionario}
	(- \Delta)^{\alpha}u = g(u), \quad \mbox{in} \quad \mathbb{R}^N, 
\end{equation}
with $\alpha \in (0,1)$ and $N \geq 1$. For more details about this subject, we would like to cite the references found in the above mentioned papers.

Recently, Alves, Duarte and Souto \cite{ADS} have proved an abstract theorem that was used to solve a large class of Berestycki-Lions type problems, which includes Anisotropic operator, Discontinuous nonlinearity, etc.. In that paper, the authors have used the deformation lemma together with a notation of Pohozaev set to prove the abstract theorem.

Initially, in the present paper, we will deal with the existence of nontrivial solution for the problem $(P)_1$, that is, 
$$ 
-\Delta_{\gamma} u + u  =  f(u) \quad  \mbox{in} \quad \mathbb{R}^{N}, \leqno{(P)_{1}}
$$
for the {\it positive mass case}. In this case, the nonlinearity $f$ satisfies the following conditions:
\begin{itemize}
	\item [$(f_1)$] $f(s)/s = o(1)$ as $s \to 0$ ;
	\item [$(f_2)$] There exists $q \in (2,2^*_{\gamma})$ such that
	$$
	\limsup_{|s| \to +\infty}\frac{|f(s)|}{|s|^{q-1}} <+\infty, \quad \forall s \in \mathbb{R},
	$$
	where $2^*_{\gamma}=\frac{2N_{\gamma}}{N_{\gamma}-2}$;
	\item [$(f_3)$] There exists $s_0>0$ such that $F(s_0)>\frac{1}{2}|s_0|^{2}$, where $ \displaystyle F(s) = \int_0^s f(t)dt$.
\end{itemize}

Using the above assumptions we are ready to state our first result.
\begin{theorem}\label{Theorem1} ({\bf The positive mass case :  $a(z)=1$}) Assume $0 \leq \gamma < 1$, $(f_1),(f_2)$ and $(f_3)$. Then problem $(P)_1$ has at least one nonnegative nontrivial weak solution.
\end{theorem}

Our next result is associated with the  {\it zero mass case} due to Berestycki and Lions \cite{berest}, more precisely, we intend to find a solution for the following class of problems 
$$ 
-\Delta_{\gamma} u =  f(u) \quad  \mbox{in} \quad \mathbb{R}^{N}, \leqno{(P)_0}
$$
where $f$ satisfies some conditions that will be fixed later on.
 
Problems this type has attracted the attention of many mathematicians with regard to the non-existence of a solution. For instance, if $\Delta_{H}$ is the Heisenberg Laplacian on $H^N=\mathbb{C}^N\times\mathbb{R}$ and $u$ is radial in the variable $x$, $u=u(|x|,y)$, then $-\Delta_{H}u =\Delta_{x}u+4|x|^ 2\partial_{yy}u$ is a Grushin operator with $\gamma=1$, and so, $(P)_0$ is strongly related to equation
\begin{equation}
-\Delta_{H} u =  f(u) \quad  \mbox{in} \quad H^{N}.
\end{equation}
Birindelli and Prajapat \cite{Birindelli&Prajapat} studied the nonexistence of cylindrical solution for above problem with $f(u)= u^p$, $0<p<Q+2/Q-2$, where $Q$ is the homogeneous dimension of $H^N$. Recently, Yu \cite{Yu2013} studied the
Liouville type theorem in the Heisenberg group for general nonlinearity. In \cite{Yu2015}, Yu  studied the nonexistence of positive solutions  for a general Grushin operator, and the main tool in this paper is the moving plane method.  Monti and Morbidelli in \cite{Monti&Morbidelli2006DMJ} investigated the uniqueness and symmetry of solutions for the problem $(P)_0$ with critical growth and introduced a suitable Kelvin-type transform in this context. Finally, related to the bounded domains, Kogoj and Lanconelli \cite{Kogoj&Lanconelli2012NA} showed regularity, existence and Pohozaev-type nonexistence results for the Dirichlet problem 
\begin{equation}
\left \{
\begin{array}{rclcl}
-\Delta_{\lambda} u & = & f(u) & \mbox{in} & \Omega, \\
u & = & 0 & \mbox{on} & \partial\Omega,
\end{array}
\right.
\end{equation}
where $\Omega$ is a bounded open subset of $\mathbb{R}^N$, for general nonlinearities $f$. 

In order to solve $(P)_0$, we are assuming the following conditions on function $f$: 
\begin{itemize}
	\item [$(f_4)$]There are $q >2^*_{\gamma}$ and $c,\delta >0$ such that  
	$$
	|f(s)| \leq c|s|^{q-1}, \quad \mbox{for} \quad  |s| \leq \delta.
	$$
	\item [$(f_5)$]  $\displaystyle \lim_{|s| \to +\infty}\frac{f(s)}{|s|^{2^{*}_{\gamma}-1}}=0.$
	\item[$(f_6)$]  There is $s_0>0$ such that $F(s_0)>0$.
\end{itemize}
 
\begin{theorem}\label{Theorem1'} ({\bf The zero mass case: $a(z)=0$}) Assume $0 \leq \gamma < 1$, $(f_4),(f_5)$ and $(f_6)$. Then problem $(P)_0$ has at least one nonnegative nontrivial weak solution.
\end{theorem}

Related to a more general class of degenerate operators, in the year $2017$, Chen, Tang and Gao studied in \cite{ChenTang&Gal2017SSMH} (see also \cite{Anh&My2016CCVE}, \cite{Kogoj&Lanconelli2012NA}, \cite{Luyen&Tri2019CVEE} and \cite{Rahal&Hamdani2018JFT}) the following problem
\begin{equation}
	\left \{
	\begin{array}{rclcl}
		-\Delta_{\lambda} u + a(z)u & = & f(x,u) & \mbox{in} & \Omega, \\
		u & = & 0 & \mbox{on} & \partial\Omega,
	\end{array}
	\right.
\end{equation}
where $\Omega$ is a bounded domain in $\mathbb{R}^{N}$, $\Delta_{\lambda}$-Laplace is given by \eqref{general_operator}, $a(z)$ is allowing to be sign-changing and $f$ is a function with a more general super-quadratic growth. 

If we  suppose that $f$ satisfies the well-known Ambrosetti-Rabinowitz superlinear condition, that is, there exists exists $\theta> 2$ such that
$$
0<\theta F(t)\leq f(t)t, ~~ t\in\mathbb{R}\backslash\{0\}, \leqno{(f_7)}
$$
we are able to establish the existence of nontrivial solution for the problem $(P)_a$ for all $\gamma\geq 0$ by supposing that %
\begin{equation}\label{A1}\tag{A1}
a(z)\geq a_0>0, ~ \hbox{for all} ~ z\in\mathbb{R}^N
\end{equation}
and 
\begin{equation}\label{A2}\tag{A2}
	a(x,y)=a(x',y) ~ \hbox{for all} ~ x, x'\in\mathbb{R}^{m} ~ \hbox{with}  ~ |x|=|x'| ~ \hbox{and all} ~ y\in\mathbb{R}^{k}.
\end{equation}

\noindent In addition to the above conditions, we will consider one of the following conditions: \\

\noindent The function $a$ is periodic on variable $y$, that is, 
\begin{equation}\label{A3}\tag{A3}
a(x,y)=a(x,y+y_0)~, ~ \hbox{for all} ~ x\in\mathbb{R}^m, y\in\mathbb{R}^k ~ \hbox{and} ~  y_0\in\mathbb{Z}^k,
\end{equation}
or the function $a$ is coercive on variable $y$, that is, 
\begin{equation}\label{A4}\tag{A4}
a(x,y)\longrightarrow +\infty ~ \hbox{when} ~ |y|\longrightarrow\infty, ~ \hbox{uniformly for} ~ x\in\mathbb{R}^m.
\end{equation}

Our main result involving the above conditions is the following: 

\begin{theorem}\label{Theorem2} ({\bf Nonconstant case}) Assume $\gamma \geq 0 $, $(f_1),(f_2)$ and $(f_7)$ and let $a:\mathbb{R}^{N}\longrightarrow\mathbb{R}$ be a continuous function satisfying \eqref{A1}  and \eqref{A2}. If \eqref{A3} or \eqref{A4} holds, then problem $(P)_a$ has at least one nonnegative nontrivial weak solution.
\end{theorem}

The Theorem \ref{Theorem2} complements the study made in Coti Zelati and Rabinowitz \cite{CotiZelati&Rabinowitz1992CPAM}, Kryszewski and Szulkin \cite{Kryszewski&Szulkin1996RRMSU} and Rabinowitz \cite{Rabinowitz1992ZAMP}, because in those papers it was considered the case where $\gamma=0$ and $a$ being a periodic function. 

Our last result is associated with the behavior of the solutions at infinite. 

\begin{theorem}\label{Theorem3} Let $u$ be a solution of problem $(P)_a$ with $a \not=0$. Then $u\in L^{\infty}(\mathbb{R}^N)$ and $\displaystyle \lim_{|z|\to\infty}u(z)=0.$  
\end{theorem}

In what follows, we say that a function $u\in\mathcal{H}_{\gamma}^{a}(\mathbb{R}^{N})$ is a {\it weak solution} to problem $(P)_a$, with $a\neq 0$, if $u\geq 0$ a.e. in $\mathbb{R}^{N}$, $ u \not \equiv 0$  and
\begin{equation}\label{solution_definition}
\int_{\mathbb{R}^{N}}\big(\nabla_{\gamma} u\nabla_{\gamma}\varphi +a(z)u\varphi \big)\,dz =\int_{\mathbb{R}^{N}} f(u)\varphi \,dz, \quad \forall \varphi\in\mathcal{H}_{\gamma}^{a}(\mathbb{R}^{N}).
\end{equation}
Similarly, $u\in\mathcal{D}_{\gamma}^{1,2}\big(\mathbb{R}^{N}\big)$ is a {\it weak solution} to problem $(P)_0$ if \eqref{solution_definition} hold for all $\varphi\in\mathcal{D}_{\gamma}^{1,2}\big(\mathbb{R}^{N}\big)$. The spaces $\mathcal{H}_{\gamma}^{1,2}(\mathbb{R}^{N})$, $\mathcal{H}_{\gamma}^{a}(\mathbb{R}^{N})$,  $\mathcal{D}_{\gamma}^{1,2}\big(\mathbb{R}^{N}\big)$ and the notation $\nabla_{\gamma}$ are defined in Section 2.

Before concluding this section, we would like to point out that in the proof of Theorems \ref{Theorem1}, \ref{Theorem1'} and \ref{Theorem2} we have found some difficulties to apply variational methods  when $\gamma>0$, because different from the case $\gamma=0$, the energy functional  associated with problem $(P)_a$ is not invariant on  the group $O(N)$ and it is also not invariant by  for some types of translations. These facts bring to the problem some technical difficulties that must be analyzed in a very cautious way. Have this in mind, in Section 2, we made a careful study about  the space functions that will be used in this paper, which includes compactness results and a version of a result due to Lions \cite[Lemma I.1]{Lions}  for the space $\mathcal{H}_{\gamma}^{1,2}(\mathbb{R}^{N})$.

Throughout the paper, unless explicitly stated, the symbol ${C}$ will always denote a generic positive constant, which may vary from line to line. The usual norm in $L^{p}(\mathbb{R}^{N})$ will be denoted by $|\cdot|_{p}$ and $H^{1}(\Omega)$ will be used to denote standard Sobolev spaces in $\Omega \subset \mathbb{R}^N$.


\section{Function spaces}

In this section we will introduce some notations and preliminary results that will be used later on.
First of all, let $\gamma$ be a nonnegative real number and $z=(x,y)=(x_1,\ldots,x_m,y_1,\ldots,y_k)\in\mathbb{R}^m\times\mathbb{R}^k$ with $m, k\geq1$ and $m+k=N\geq 3$. In what follows, we designate by $\mathcal{H}_{\gamma}^{1,2}\big(\mathbb{R}^{N}\big)$ the weighted Sobolev space 
\begin{equation*}
\mathcal{H}_{\gamma}^{1,2}\big(\mathbb{R}^{N}\big)=\left\{u\in L^{2}\big(\mathbb{R}^{N}\big):\frac{\partial u}{\partial x_i}, ~|x|^{\gamma}\frac{\partial u}{\partial y_j}\in L^{2}\big(\mathbb{R}^{N}\big), i=1,\ldots,m, j=1,\ldots,k.\right\}.
\end{equation*}

\noindent For simplicity, if  $u\in \mathcal{H}_{\gamma}^{1,2}(\mathbb{R}^{N})$, we denote by $\nabla_{\gamma}$  the gradient operator defined by
\begin{equation*}
\nabla_{\gamma}u=\left(\nabla_{x}u, |x|^{\gamma}\nabla_{y} u \right)= \left(u_{x_1},\cdots,u_{x_m},  |x|^{\gamma}u_{y_1},\cdots,|x|^{\gamma}u_{y_k} \right),
\end{equation*}
and so,
\begin{equation*}
|\nabla_{\gamma}u|^{2}=|\nabla_{x}u|^2+|x|^{2\gamma}|\nabla_{y}u|^2.
\end{equation*}
We would like to mention  that $\mathcal{H}_{\gamma}^{1,2}\big(\mathbb{R}^{N}\big)$ is a Hilbert space, when endowed with the scalar product given by
\begin{equation*}
\langle u,v \rangle_{\mathcal{H}_{\gamma}^{1,2}(\mathbb{R}^{N})} = \int_{\mathbb{R}^{N}}\big(\nabla_{\gamma} u\nabla_{\gamma} v + uv \big) \,dz,
\end{equation*}
whose the corresponding norm is 
\begin{equation*}
||u||_{\gamma}:=\sqrt{\langle u,u \rangle_{\mathcal{H}_{\gamma}^{1,2}(\mathbb{R}^{N})}}=\left(\int_{\mathbb{R}^{N}}\big(|\nabla_{\gamma}u|^2+|u|^2\big) \, dz\right)^{\frac{1}{2}}.
\end{equation*}
By Sobolev inequality found in Monti \cite{Monti2006CPDE}, we have
\begin{equation}\label{Sobolev-inequality}
\left(\int_{\mathbb{R}^{N}}|u|^{2_{\gamma}^{*}} \, dz\right)^{2/2_{\gamma}^{*}}\leq C(m,k,\gamma)\int_{\mathbb{R}^{N}}|\nabla_{\gamma}u|^2 \, dz,
\end{equation}
from where it follows that the embedding $\mathcal{H}_{\gamma}^{1,2}\big(\mathbb{R}^{N}\big) \hookrightarrow L^{2_{\gamma}^{*}}\big(\mathbb{R}^{N}\big)$ is continuous, then there is $\mathcal{C}>0$ such that
\begin{equation*}
|u|_{2_{\gamma}^{*}}\leq \mathcal{C}||u||_{\gamma}, \quad \forall u \in \mathcal{H}_{\gamma}^{1,2}\big(\mathbb{R}^{N}\big).
\end{equation*}
Since 
\begin{equation*}
|u|_{2}\leq ||u||_{\gamma}, \quad \forall u \in \mathcal{H}_{\gamma}^{1,2}\big(\mathbb{R}^{N}\big),
\end{equation*}
the interpolation inequality on the Lebesgue space ensures that the embedding \linebreak $\mathcal{H}_{\gamma}^{1,2}\big(\mathbb{R}^{N}\big) \hookrightarrow L^{p}\big(\mathbb{R}^{N}\big)$ is continuous for every $p\in [2, 2^{*}_{\gamma}]$. Thus, for each $p \in [2, 2^{*}_{\gamma}]$, there is $\mathcal{C}_p>0$ such that
\begin{equation}\label{embedding-continuous}
|u|_{p}\leq \mathcal{C}_p||u||_{\gamma},  \quad \forall u \in \mathcal{H}_{\gamma}^{1,2}\big(\mathbb{R}^{N}\big).
\end{equation}

Here, we would like point out that if $\Omega_1 \subset \mathbb{R}^m$ and $\Omega_2 \subset \mathbb{R}^k$ are smooth bounded domains, the embeddings (\ref{embedding-continuous}) still hold by replacing  $\mathbb{R}^N$ for any set of the form $\Omega=\mathbb{R}^m \times \Omega_2 $, $\Omega=\Omega_1 \times \mathbb{R}^k $ or $\Omega=\Omega_1 \times \Omega_2 $, that is,  
\begin{equation}\label{embedding-continuous1}
|u|_{L^{p}(\Omega)}\leq \mathcal{C}_p||u||_{\mathcal{H}_{\gamma}^{1,2}(\Omega)},  \quad \forall u \in \mathcal{H}_{\gamma}^{1,2}(\Omega),
\end{equation}
where 
\begin{equation*}
\mathcal{H}_{\gamma}^{1,2}\big(\Omega\big)=\left\{u\in L^{2}\big(\Omega\big):\frac{\partial u}{\partial x_i}, ~|x|^{\gamma}\frac{\partial u}{\partial y_j}\in L^{2}\big(\Omega\big), i=1,\ldots,m, j=1,\ldots,k.\right\}
\end{equation*} 
and 
\begin{equation*}
||u||_{\mathcal{H}_{\gamma}^{1,2}(\Omega)}:=\left(\int_{\Omega}\big(|\nabla_{\gamma}u|^2+|u|^2\big) \, dz\right)^{\frac{1}{2}}.
\end{equation*}

The proof of \eqref{embedding-continuous1} follows by showing the existence of an Extension Operator, which can be done as in Br\'ezis \cite[Section 9.2]{Brezis}. Moreover, when $\Omega=\Omega_1 \times \Omega_2 \subset \mathbb{R}^N=\mathbb{R}^m \times \mathbb{R}^k $, the embedding $\mathcal{H}_{\gamma}^{1,2}(\Omega) \hookrightarrow L^{p}(\Omega)$ is compact for every $p\in [1, 2^{*}_{\gamma})$, which is a consequence of the Strauss's Compactness Lemma \cite[Theorem A.I]{berest}.

Setting the space
\begin{equation*}
\mathcal{H}_{\gamma}^{a}(\mathbb{R}^N)=\left\{u\in\mathcal{H}_{\gamma}^{1,2}\big(\mathbb{R}^{N}\big):\int_{\mathbb{R}^{N}}a(z)|u|^2 \, dz <\infty\right\},
\end{equation*}
it is easy to see that condition \eqref{A1} implies that $\mathcal{H}_{\gamma}^{a}(\mathbb{R}^N)$ is Hilbert space with scalar product and norm given by
\begin{equation*}
\langle u,v \rangle_{a} = \int_{\mathbb{R}^{N}}\big(\nabla_{\gamma} u\nabla_{\gamma} v + a(z)uv \big) \,dz ~~ \hbox{and} ~~ ||u||^{2}:=\sqrt{\langle u,u \rangle_{a}}.
\end{equation*}
Moreover,  \eqref{A1} also yields $\mathcal{H}_{\gamma}^{a}(\mathbb{R}^N)$ is embedded continuously into $\mathcal{H}_{\gamma}^{1,2}(\mathbb{R}^{N})$. Hence, by (\ref{embedding-continuous}),   $\mathcal{H}_{\gamma}^{a}(\mathbb{R}^N)$ is embedded continuously in $L^{p}(\mathbb{R}^{N})$ for every $p\in [2, 2^{*}_{\gamma}]$.\\

Another important space that  will consider in the present paper is $\mathcal{D}_{\gamma}^{1,2}(\mathbb{R}^{N}) $, which is the closure of $C_{0}^{\infty}(\mathbb{R}^N)$ with relation to the norm
$$
\|\varphi \|_{\mathcal{D}_{\gamma}^{1,2}(\mathbb{R}^{N})}=\left(\int_{\mathbb{R}^{N}}|\nabla_{\gamma}u|^2 \, dz\right)^{\frac{1}{2}}, \quad \varphi \in C_{0}^{\infty}(\mathbb{R}^N).
$$
From (\ref{Sobolev-inequality}), we derive that the embedding $\mathcal{D}_{\gamma}^{1,2}(\mathbb{R}^{N}) \hookrightarrow L^{2^*_{\gamma}}(\mathbb{R}^N)$ is continuous, that is, there is $\mathcal{S}>0$ such that 
\begin{equation} \label{S}
|u|_{2^*_{\gamma}}\leq \mathcal{S}\|u\|_{\mathcal{D}_{\gamma}^{1,2}(\mathbb{R}^{N})},  \quad \forall u \in \mathcal{D}_{\gamma}^{1,2}(\mathbb{R}^{N}).
\end{equation}
\subsection{Compactness Results}

In this section we will prove some important compactness results that will be used in the proof of Theorems \ref{Theorem1}, \ref{Theorem1'} and \ref{Theorem2}. 

Our first result is a version of a famous result due to Lions \cite[Lemma I.1]{Lions} for the space  $\mathcal{H}_{\gamma}^{1,2}(\mathbb{R}^{N})$.

\begin{lemma}\label{Lions-lemma} 
	Let $r>0, 2\leq q < 2^{*}_{\gamma}$ and $(u_n) \subset \mathcal{H}_{\gamma}^{1,2}(\mathbb{R}^{N})$ be a bounded sequence with
	\begin{equation*}
		\sup_{y\in\mathbb{R}^{k}}\int_{\mathbb{R}^{m}\times B^{k}_r(y)}|u_n|^{q}dz \longrightarrow 0 \quad \mbox{as} \quad n\longrightarrow\infty,
	\end{equation*}	
where $ B_r^{k}(y)=\{\xi \in \mathbb{R}^k\,:\,|\xi - y|<r\}.$ Then, $u_n\longrightarrow 0$ in $L^{p}(\mathbb{R}^{N})$ for $2<p<2^{*}_{\gamma}$.	
\end{lemma}
\begin{proof}
	First all, let us consider
	$$
	\mathbb{R}^{N}=\mathbb{R}^{m}\times\mathbb{R}^{k}\subset\bigcup_{y\in\mathbb{R}^{k}}\left(\mathbb{R}^{m}\times B_r(y)\right),
	$$
	in such way that each point of $\mathbb{R}^k$ is contained in a finite quantity $l$ of balls of radius $r$. Now, let $q<s<2^{*}_{\gamma}$ and $\alpha:=\frac{(s-q)2^{*}_{\gamma}}{(2^{*}_{\gamma}-q)s}$. Then, if $u\in \mathcal{H}_{\gamma}^{1,2}(\mathbb{R}^{N})$, the embedding  \eqref{embedding-continuous1} together with the interpolation inequality on the Lebesgue space says that 
	\begin{eqnarray*}
		|u|_{L^{s}(\mathbb{R}^{m}\times B_r(y))} & \leq & |u|^{1-\alpha}_{L^{q}(\mathbb{R}^{m}\times B_r(y))}|u|^{\alpha}_{L^{2^{*}_{\gamma}}(\mathbb{R}^{m}\times B_r(y))} \\
		&\leq&\mathcal{C}|u|^{1-\alpha}_{L^{q}(\mathbb{R}^{m}\times B_r(y))}\left(\int_{\mathbb{R}^{m}\times B_r(y)}\big(|\nabla_{\gamma}u|^2+|u|^{2}\big) \, dz\right)^{\frac{\alpha}{2}}. \\
	\end{eqnarray*}
	Choosing $\alpha=2/s$, we obtain the inequality
	\begin{equation*}
		\int_{\mathbb{R}^{m}\times B_r(y)}|u|^{s} \, dz\leq \mathcal{C}^s|u|^{(1-\alpha)s}_{L^{q}(\mathbb{R}^{m}\times B_r(y))}\int_{\mathbb{R}^{m}\times B_r(y)}\big(|\nabla_{\gamma}u|^2+|u|^{2}\big) \, dz,
	\end{equation*}
	which leads to 
	\begin{equation*}
		\int_{\mathbb{R}^{N}}|u|^{s} \, dz \leq \mathcal{C}^sl\sup_{y\in\mathbb{R}^k}\big(\int_{\mathbb{R}^{m}\times B_r(y)}|u|^q \, dz\big)^{(1-\alpha)s/q}\int_{\mathbb{R}^{N}}\big(|\nabla_{\gamma}u|^2+|u|^{2}\big) \, dz.
	\end{equation*}
	As $(u_n)$ is bounded sequence in $\mathcal{H}_{\gamma}^{1,2}(\mathbb{R}^{N})$, the last inequality allows us to conclude that $u_n\to 0$ in $L^{s}(\mathbb{R}^{N})$. Now, applying again the interpolation inequality on the Lebesgue spaces, we conclude that $u_n \to 0$ in $L^{p}(\mathbb{R}^{N})$ for $2<p<2^{*}_{\gamma}$.
\end{proof}

Next, we will show some compactness results that are crucial in our approach. Have this in mind, we will fix the following spaces:
$$
\mathcal{H}_{\gamma,rad_{x,y}}^{1,2}\big(\mathbb{R}^{N}\big)=\left\{\mathcal{H}_{\gamma}^{1,2}\big(\mathbb{R}^{N}\big):u(x,y)=u(|x|,|y|) ~ \hbox{for all} ~ (x,y) \in \mathbb{R}^{m} \times \mathbb{R}^{k}\right\},
$$
$$
 \mathcal{H}_{\gamma,rad_x}^{1,2}\big(\mathbb{R}^{N}\big)=\left\{\mathcal{H}_{\gamma}^{1,2}\big(\mathbb{R}^{N}\big):u(x,y)=u(|x|,y) ~ \hbox{for all} ~ (x,y) \in \mathbb{R}^{m} \times  ~ \mathbb{R}^{k}\right\},
$$
$$
\mathcal{H}_{\gamma,rad_x}^{1,2}(\mathbb{R}^m \times B_r(0))=\left\{\mathcal{H}_{\gamma}^{1,2}\big(\mathbb{R}^m \times B_r(0)):u(x,y)=u(|x|,y) ~ \hbox{for all} ~ (x,y) \in \mathbb{R}^{m} \times B_r(0)\right\},
$$
$$
\mathcal{H}_{\gamma, rad_x}^{a}\big(\mathbb{R}^{N}\big)=\left\{\mathcal{H}_{\gamma}^{a}\big(\mathbb{R}^{N}\big):u(x,y)=u(|x|,y) ~ \hbox{for all} ~ (x,y) \in \mathbb{R}^{m} \times  ~ \mathbb{R}^{k}\right\}
$$
and
$$
\mathcal{D}_{\gamma,rad_{x,y}}^{1,2}\big(\mathbb{R}^{N}\big)=\left\{\mathcal{D}_{\gamma}^{1,2}\big(\mathbb{R}^{N}\big):u(x,y)=u(|x|,|y|) ~ \hbox{for all} ~ (x,y) \in \mathbb{R}^{m} \times \mathbb{R}^{k}\right\}.
$$

\begin{lemma} \label{embedding-compact}
	Let $p\in(2,2^{*}_{\gamma})$ and $r>0$. Then, the embedding of $\mathcal{H}_{\gamma,rad_x}^{1,2}(\mathbb{R}^{m}\times B_r(0))$ into $L^{p}(\mathbb{R}^{m}\times B_r(0))$ is compact.
\end{lemma}
\begin{proof}
	It is enough to show that if $(u_n)$ is a sequence in $\mathcal{H}_{\gamma,rad_x}^{1,2}(\mathbb{R}^{m}\times B_r(0))$ such that $u_n \rightharpoonup 0$ in  $\mathcal{H}_{\gamma,rad_x}^{1,2}(\mathbb{R}^{m}\times B_r(0))$, then $u_n\to 0$ in $L^{p}(\mathbb{R}^{m}\times B_r(0))$. 
	
	Let $\delta >0$ and $h \in C^{\infty}(\mathbb{R},[0,1])$ defined by
	\begin{equation*}
		h(t)=\left \{
		\begin{array}{ccl}
			0 & \mbox{if}  & |t|\leq \delta/2, \\
			1 & \mbox{if} &  |t|\geq \delta. \\
		\end{array}
		\right.
	\end{equation*}
	Notice that if $u\in \mathcal{H}_{\gamma}^{1,2}(\mathbb{R}^{m}\times B_r(0))$ and
	$$
	v(z):=h(|x|^2)u(z), ~~ z=(x,y)\in\mathbb{R}^{m}\times B_r(0),
	$$
	then $|v|^{2}\leq|u|^{2}$, 
	\begin{equation*}
		\left|\frac{\partial v}{\partial x_i}\right|^2\leq {C}\big(|u|^2+\left|\frac{\partial u}{\partial x_i}\right|^2 \big), ~~ i=1,\ldots,m
	\end{equation*}
and
	\begin{equation*}
		\left|\frac{\partial v}{\partial y_j}\right|^2 = h^2(|x|^2)\left|\frac{\partial u}{\partial y_j}\right|^2, ~~ j=1,\ldots,k.
	\end{equation*}
Hence,
	\begin{eqnarray*}
		\int_{\mathbb{R}^{m}\times B_r(0)}|\nabla_{y}v|^2dz & = & \int_{[|x|\geq\delta/2]\times B_r(0)}h^2(|x|^2)|\nabla_{y}u|^2dz \\
		&\leq& {C}\int_{[|x|\geq\delta/2]\times B_r(0)}|x|^{2\gamma}|\nabla_{y}u|^2dz \\
		&\leq& {C}\int_{\mathbb{R}^{m}\times B_r(0)}|x|^{2\gamma}|\nabla_{y}u|^2dz. \\
	\end{eqnarray*}
Therefore,  $ v\in H^{1}(\mathbb{R}^{m}\times B_r(0))$ and
	\begin{equation*}
		||v||_{H^{1}(\mathbb{R}^{m}\times B_r(0))}\leq {C}||u||_{\mathcal{H}_{\gamma}^{1,2}(\mathbb{R}^{m}\times B_r(0))}.
	\end{equation*}
	Thereby, setting  $v_n(x,y)=h(|x|^2)u_n(x,y)$, we have  $(v_n)$ is a bounded sequence in \linebreak $H_{rad,x}^{1}(\mathbb{R}^{m}\times B_r(0))$ and $v_n \rightharpoonup 0$ in $H_{rad,x}^{1}(\mathbb{R}^{m}\times B_r(0))$. By Lions \cite[Lemme III.2 ]{Lions1982JFA}, 
	\begin{equation*}
		v_n\to 0 ~~	\hbox{in} ~~ L^{p}(\mathbb{R}^{m}\times B_r(0)), ~~ 2<p<2^{*}.
	\end{equation*}
	Since $v_n(x,y)=u_n(x,y)$ for $(x,y) \in [|x| \geq \delta] \times B_r(0)$, we derive that
		\begin{equation}\label{conv1}
	u_n\to 0 ~~	\hbox{in} ~~ L^{p}([|x| \geq  \delta] \times B_r(0)), ~~ 2<p<2^{*}.
	\end{equation}

	Furthermore, recalling that if $\Omega=B_\delta(0) \times B_r(0) \subset \mathbb{R}^N$ the embedding $\mathcal{H}_{\gamma}^{1,2}(\Omega)$ into $L^p(\Omega)$ is compact for $1\leq p<2_{\gamma}^{*}$, then
	\begin{equation}\label{conv2}
		u_n\to 0 ~~	\hbox{in} ~~ L^{p}([|x|\leq\delta]\times B_r(0)).
	\end{equation}
	Accordingly to  \eqref{conv1} and \eqref{conv2}, $u_n\to 0$	in $L^{p}(\mathbb{R}^{m}\times B_r(0))$ with  $p\in(2,2^{*}_{\gamma})$.
\end{proof}

\begin{lemma} \label{embedding-compact0}
	Let $p\in(2,2^{*}_{\gamma})$. Then, the embedding of $\mathcal{H}_{\gamma,rad_{x,y}}^{1,2}(\mathbb{R}^{N}) \hookrightarrow L^{p}(\mathbb{R}^{N})$ is compact.
\end{lemma}
\begin{proof} In what follows, we set $q \in (2,2^*_{\gamma})$ and
$$
k(y,r,O(k))=\{n \in \mathbb{N}\,:\, \exists g_1,g_2,...,g_n \in O(k)\,:\,B^{k}_{r}(g_i(y))\cap B^{k}_r(g_j(y))=\emptyset\}.
$$	
From definition of $k(y,r,O(k))$, 
$$
\lim_{|y| \to +\infty}k(y,r,O(k))=+\infty.
$$	
A simple computation gives that for each $u \in \mathcal{H}_{\gamma,rad_{x,y}}^{1,2}(\mathbb{R}^{N})$, 
$$
\int_{\mathbb{R}^m \times B^{k}_{r}(y)}|u|^{q}\,dz \leq \frac{|u|_q^{q}}{k(y,r,O(k))}.
$$
Arguing as in the last lemma, let  $(u_n)$ be a sequence in $\mathcal{H}_{\gamma,rad_{x,y}}^{1,2}(\mathbb{R}^{N})$ such that $u_n \rightharpoonup 0$ in  $\mathcal{H}_{\gamma,rad_{x,y}}^{1,2}(\mathbb{R}^{N})$. As $(u_n)$ is a bounded sequence in $\mathcal{H}_{\gamma,rad_{x,y}}^{1,2}(\mathbb{R}^{N})$, given $\epsilon>0$, the definition of $k(y,r,O(k))$ ensures that there is $R>0$ such that
$$
\sup_{|y| \geq R}\int_{\mathbb{R}^m \times B^{k}_{r}(y)}|u_n|^{q}\,dz< \frac{\epsilon}{2}, \quad \forall n \in\mathbb{N}.
$$
On the other hand, the Lemma \ref{embedding-compact} yields
$$
\int_{\mathbb{R}^m \times B_{R+r}(0)}|u_n|^{q}\,dz \to 0, 
$$
and so,
$$
\sup_{y \in \mathbb{R}^k}\int_{\mathbb{R}^m \times B^{k}_{r}(y)}|u_n|^{q}\,dz \to 0. 
$$
Now the lemma follows by employing Lemma \ref{Lions-lemma}. 
\end{proof}

\begin{cor} If \eqref{A1} holds, then the embedding $\mathcal{H}^{a}_{\gamma,rad_{x,y}}(\mathbb{R}^{N}) \hookrightarrow L^{p}(\mathbb{R}^{N})$ is compact.
	
\end{cor}	
\begin{proof} The corollary is an immediate consequence of the last lemma, because we have the continuous embedding of $\mathcal{H}_{\gamma}^{a}(\mathbb{R}^N)$ into $\mathcal{H}_{\gamma}^{1,2}(\mathbb{R}^{N})$.
	
\end{proof}

\begin{lemma} \label{embedding-compact01}
	Let  $(u_n) \subset \mathcal{D}_{\gamma,rad_{x,y}}^{1,2}(\mathbb{R}^{N})$ be a bounded sequence with $u_n \rightharpoonup 0$ in $\mathcal{D}_{\gamma,rad_{x,y}}^{1,2}(\mathbb{R}^{N})$. Then, for each $0<a<b<+\infty$ and $p \in (2,2^{*}_{\gamma})$,  
	$$
	\int_{\{a<|u_n|<b\}}|u_n|^{p}\,dz \to 0. 
	$$
\end{lemma}
\begin{proof} In the sequel, let us consider $g \in C^{\infty}(\mathbb{R},\mathbb{R})$ satisfying
$$
g(t)=0 \quad \mbox{for} \quad |t| \leq \frac{a}{2} \quad \mbox{or} \quad |t| \geq b+1, 
$$	 
and 
$$
g(t)=t \quad \mbox{for} \quad a\leq |t| \leq b.  
$$	
Using the function $g$, we set the sequence $v_n(x,y)=g(u_n(x,y))$. A direct computation shows that $v_n \in \mathcal{H}_{\gamma,rad_{x,y}}^{1,2}(\mathbb{R}^{N})$, and that it is a bounded sequence in $\mathcal{H}_{\gamma,rad_{x,y}}^{1,2}(\mathbb{R}^{N})$. Thereby, by Lemma \ref{embedding-compact0}, $v_n \to 0$ in $L^{p}(\mathbb{R}^N)$, that is,
$$
\int_{\mathbb{R}^N}|v_n|^{p}\,dz \to 0,
$$
and so, 
$$
\int_{\{a<|u_n|<b\}}|v_n|^{p}\,dz \to 0.
$$
Using the fact that $g(t)=t$ for $a \leq |t| \leq b$, it follows that
$$
\int_{\{a<|u_n|<b\}}|u_n|^{p}\,dz \to 0,
$$
showing the lemma.

\end{proof}

\begin{lemma} \label{embedding-compact-coercive}
Assume that \eqref{A1} and \eqref{A4} hold. Then the embedding of $\mathcal{H}_{\gamma,rad_x}^{a}(\mathbb{R}^{N}) \hookrightarrow L^{p}(\mathbb{R}^{N})$ is compact for all $p\in[2,2^{*}_{\gamma})$.
\end{lemma}
\begin{proof} 
	It suffices to show that for any sequence  $(u_n) \subset \mathcal{H}_{\gamma, rad_x}^{a}(\mathbb{R}^{N})$ with $u_n \rightharpoonup 0$ in $\mathcal{H}_{\gamma, rad_x}^{a}(\mathbb{R}^{N})$, we must have $u_n\to 0$ in $L^{p}(\mathbb{R}^{N})$. By \eqref{A4}, given $\varepsilon>0$, we can take $R_{\varepsilon}>0$ such that
	$$ 
	\frac{1}{a(x,y)}\leq\frac{\varepsilon}{2}, ~~ \hbox{for all} ~~ (x,y)\in\mathbb{R}^{m}\times [|y|\geq R_{\varepsilon}].
	$$
From this, 
	$$
	\int_{\mathbb{R}^{m}\times [|y|\geq R_{\varepsilon}]} |u_n|^2  \, dz \leq\frac{\varepsilon}{2}\int_{\mathbb{R}^{m}\times [|y|\geq R_{\varepsilon}]} a(x,y)|u_n|^2 \, dz < \varepsilon M, 
	$$
	where $M=\sup_{n}||u_n||^2_a$. 
	
	Notice that \eqref{A1} and \eqref{embedding-continuous1} imply that the space $\mathcal{H}_{\gamma, rad_x}^{a}(\mathbb{R}^{m}\times [|y|\geq R_{\varepsilon}])$ is continuously immersed in $L^{2^{*}_{\gamma}}(\mathbb{R}^{m}\times [|y|\geq R_{\varepsilon}])$. Then, the last inequality together with the interpolation inequality ensures that for each  $p\in(2, 2^{*}_{\gamma})$, 
	$$
	\int_{\mathbb{R}^{m}\times [|y|\geq R_{\varepsilon}]} |u_n|^p \, dz < \varepsilon^{\frac{(1-\alpha)p}{2}} \mathcal{C}_p,
	$$
	where $\alpha=\frac{(p-2)2^{*}_{\gamma}}{(2^{*}_{\gamma}-2)p}$. 
	
	Once again, by \eqref{A1} and Lemma \ref{embedding-compact} it is easy to conclude that the the embedding of $\mathcal{H}_{\gamma, rad_x}^{a}(\mathbb{R}^{m}\times B_{R_{\varepsilon}}(0))$ into $L^{p}(\mathbb{R}^{m}\times B_{R_{\varepsilon}}(0))$ is compact for all $p\in[2,2^{*}_{\gamma})$, and so, $u_n\to 0$ in $L^{p}(\mathbb{R}^{m}\times B_{R_{\varepsilon}}(0))$. This proves the desired result.
	
\end{proof}

\section{Proof of Theorem \ref{Theorem1}}

From now on, since we intend to find nonnegative solutions, without loss of generality we assume that 
$$
f(s)=0, \quad \forall s \leq 0.
$$

The energy functional associated with problem $(P)_1$ is the functional  $I:\mathcal{H}_{\gamma}^{1,2}(\mathbb{R}^{N})\longrightarrow\mathbb{R}$ given by 
\begin{equation*}
I(u)=\frac{1}{2}\int_{\mathbb{R}^{N}}\big(|\nabla_{\gamma}u|^2 +|u|^2\big) \, dz - \int_{\mathbb{R}^{N}} F(u)\, dz.
\end{equation*}
The reader is invited to see that weak solutions of problem $(P)_1$ are critical points of $I$. 

In order to overcome the loss of compactness involving the space $\mathcal{H}_{\gamma}^{1,2}(\mathbb{R}^{N})$, we will restrict to functional $I$ to the space $\mathcal{H}_{\gamma,rad_{x,y}}^{1,2}(\mathbb{R}^{N})$, and this sense, an important tool in our approach is the {\it Principle of Symmetric Criticality} due to Palais (See book \cite[Theorem 1.28]{Willem1996}), which will be used to obtain solutions on the whole space $\mathcal{H}_{\gamma}^{1,2}(\mathbb{R}^{N})$.

Our first lemma establishes that $I$ satisfies the mountain pass geometry.

\begin{lemma}\label{MP-geometry} There exist positive numbers $\rho$ and $\alpha$ such that:
	\begin{description}
		\item[i)] $I(u)\geq\alpha$ if  $||u||_{\gamma}=\rho$.
		\item[ii)] There exists $v\in \mathcal{H}_{\gamma,rad_{x,y}}^{1,2}(\mathbb{R}^{N})$ such that $||v||_{\gamma}>\rho$ and $I(v)< 0$.
	\end{description}
\end{lemma}

\begin{proof} {\bf i)} From $(f_1)$ and $(f_2)$, there is a constant $C>0$ such that
	\begin{equation*}
	|F(t)|\leq \frac{1}{4}|t|^2+C|t|^{q}, ~~ \hbox{for all} ~~ t\in\mathbb{R}.
	\end{equation*}
	Hence,
$$
		I(u)  =  \frac{1}{2}||u||_{\gamma}^2 -\int_{\mathbb{R}^{N}}F(u)dz \geq \frac{1}{2}||u||_{\gamma}^2 - \frac{1}{4}\int_{\mathbb{R}^{N}}|u|^2\,dz-C\int_{\mathbb{R}^{N}}|u|^q\,dz. 
$$
	Since the embedding $\mathcal{H}_{\gamma,rad_{x,y}}^{1,2}(\mathbb{R}^{N}) \hookrightarrow L^{q}(\mathbb{R}^{N})$ is continuous,
	$$
	I(u)\geq \frac{1}{2}||u||_{\gamma}^2-\frac{1}{4}||u||_{\gamma}^2-C||u||_{\gamma}^{q},
	$$
and so, 
	$$
	I(u)\geq \frac{1}{4}||u||_{\gamma}^2-C||u||_{\gamma}^{q}.
	$$
	Now, if $||u||_{\gamma}=\rho$ is small enough, it follows that 
	$$
	I(u)\geq \frac{1}{4}\rho^2-C\rho^{q}:=\alpha>0.
	$$
	
	\noindent {\bf ii)} By $(f_3)$, we know that there is a nonnegative function $\varphi \in C_0^{\infty}(\mathbb{R}^N) \cap \mathcal{H}_{\gamma,rad_{x,y}}^{1,2}(\mathbb{R}^{N}) \setminus \{0\}$ satisfying 
\begin{equation} \label{EQ1}
	\displaystyle \int_{\mathbb{R}^N}(F(\varphi)-\frac{1}{2}|\varphi|^2)\,dz>0. 
\end{equation}	
By using the function $\varphi$, for each $t>0$ we fix the function 
	$$
	\varphi_t(z)=\varphi(z/t), \quad  \forall z \in \mathbb{R}^N. 
	$$
	A direct computation leads to
	$$
	I(\varphi_t)=\frac{t^{N-2}}{2}\int_{\mathbb{R}^{N}}|\nabla_{x}\varphi|^2\,dz +\frac{t^{N+2\gamma-2}}{2}\int_{\mathbb{R}^{N}}|x|^{2\gamma}|\nabla_{y}\varphi|^2\,dz+ t^{N}\left(\frac{1}{2}\int_{\mathbb{R}^{N}}|\varphi|^2 \, dz - \int_{\mathbb{R}^{N}} F(\varphi)\, dz\right).
	$$
Now, the fact that $0\leq \gamma<1$ combines with (\ref{EQ1}) to give
$$
I(\varphi_t) \to -\infty \quad \mbox{as} \quad t \to +\infty.
$$	
Therefore, the result follows by fixing $v=\varphi_t$ for $t$ large enough.	
\end{proof}

The last lemma permits to conclude that the mountain pass level given by 
\begin{equation*}
d:=\inf_{g\in\Gamma}\max_{t\in [0,1]}I(g(t))\geq\alpha>0,
\end{equation*}
where 
\begin{equation*}
\Gamma=\{\xi \in C([0,1],\mathcal{H}_{\gamma,rad_{x,y}}^{1,2}(\mathbb{R}^{N})):\xi(0)=0 ~ \hbox{and} ~ \xi(1)=v \},
\end{equation*}
is well defined.

Since we are not assuming in this section the Ambrosetti-Rabinowitz condition, we will adapt for our problem the same approach explored in Jeanjean \cite{Jeanjean}, which was also used in 	Hirata, Ikoma and Tanaka \cite{JNK}.

Hereafter, let us consider an auxiliary functional $\tilde{I}:\mathbb{R} \times \mathcal{H}_{\gamma,rad_{x,y}}^{1,2}(\mathbb{R}^{N}) \to \mathbb{R}$ defined by 
$$
\tilde{I}(s,u(z))=I(u(e^{-s}x,e^{-(\gamma+1)s}y)), \quad z=(x,y) \in \mathbb{R}^N=\mathbb{R}^m \times \mathbb{R}^k,
$$
that is,
$$
\tilde{I}(s,u)= \frac{e^{(m+(\gamma+1)-2)s}}{2}\|u\|_{\mathcal{D}_{\gamma}^{1,2}\big(\mathbb{R}^{N}\big)}^{2}+\frac{e^{(m+(\gamma+1)k)s}}{2}|u|_{2}^{2}-e^{(m+(\gamma+1)k)s}\int_{\mathbb{R}^N}F(u)\,dz.
$$
It is easy to check that 
$$
\tilde{I}(0,u)=I(u), \quad \forall u \in \mathcal{H}_{\gamma,rad_{x,y}}^{1,2}(\mathbb{R}^{N}).
$$

In what follows, we equip the space $\mathbb{R} \times\mathcal{H}_{\gamma,rad_{x,y}}^{1,2}(\mathbb{R}^{N})$ with the standard product norm $\|(s,u)\|_{\mathbb{R} \times \mathcal{H}_{\gamma}^{1,2}(\mathbb{R}^{N})}=\left(|s|^{2}+\|u\|_{\gamma}^{2}\right)^{\frac{1}{2}}$.

Arguing of the same way as in Lemma \ref{MP-geometry}, we have that $\tilde{I}$ also satisfies the mountain pass geometry, more precisely, we have the result below:

\begin{lemma}\label{MP-geometry2} There exist positive numbers $\rho$ and $\alpha$ such that:
	\begin{description}
		\item[i)] $\tilde{I}(s,u)\geq\alpha$ if  $\|(s,u)\|_{\mathbb{R} \times \mathcal{H}_{\gamma,rad_{x,y}}^{1,2}(\mathbb{R}^{N})}=\rho$.
		\item[ii)] There exists $v\in \mathcal{H}_{\gamma}^{1,2}(\mathbb{R}^{N})$ such that $\|(0,v)\|_{\mathbb{R} \times \mathcal{H}_{\gamma,rad_{x,y}}^{1,2}(\mathbb{R}^{N})}=\|v\|>\rho$ and $\tilde{I}(0,v)=I(v)< 0$.
	\end{description}
\end{lemma}

From now on, we denote by $\tilde{d}$, the mountain pass level associated with $\tilde{I}$ given by 
\begin{equation*}
\tilde{d}:=\inf_{\tilde{g}\in\Gamma}\max_{t\in [0,1]}\tilde{I}(\tilde{g}(t))\geq\alpha>0,
\end{equation*}
where 
\begin{equation*}
\Gamma=\{\tilde{g} \in C([0,1],\mathbb{R} \times \mathcal{H}_{\gamma,rad_{x,y}}^{1,2}(\mathbb{R}^{N})\,:\,\tilde{g}(0)=0 ~ \hbox{and} ~ \tilde{g}(1)=(0,v) \}.
\end{equation*}

From definition of $I$ and $\tilde{I}$, we derive the equality $d=\tilde{d}$, which permits to repeat the same arguments explored in the proof of \cite[Propostion 4.2]{JNK} ( see also  \cite[Propostion 2.2]{Jeanjean}) to show the following proposition:

\begin{prop} \label{P1} There is a sequence $\{(s_n,u_n)\} \subset \mathbb{R} \times \mathcal{H}_{\gamma,rad_{x,y}}^{1,2}(\mathbb{R}^{N})$ such that
\begin{itemize}
\item[(a)] $s_n \to 0;$ 	
\item[(b)] $\tilde{I}(s_n,u_n) \to \tilde{d}; $
\item[(c)] $\frac{\partial}{\partial u}\tilde{I}(s_n,u_n)\to 0$ strongly in $(\mathcal{H}_{\gamma,rad_{x,y}}^{1,2}(\mathbb{R}^{N}))^*$;
\item[(d)] $\frac{\partial}{\partial s}\tilde{I}(s_n,u_n)\to 0$ in $\mathbb{R}$.
\end{itemize}		
	
\end{prop}	

The next lemma establishes the boundedness of the sequence $(u_n) \subset \mathcal{H}_{\gamma,rad_{x,y}}^{1,2}(\mathbb{R}^{N}) $ that was obtained in the last proposition.

\begin{lemma} \label{limitacao} The sequence $(u_n) \subset \mathcal{H}_{\gamma,rad_{x,y}}^{1,2}(\mathbb{R}^{N}) $ is bounded. 
	
\end{lemma}	 
\begin{proof} From Proposition \ref{P1}, 
$$
\left\{ 
\begin{array}{l}
\frac{e^{(m+(\gamma+1)-2)s_n}}{2}\|u_n\|_{\mathcal{D}_{\gamma}^{1,2}\big(\mathbb{R}^{N}\big)}^{2}+\frac{e^{(m+(\gamma+1)k)s_n}}{2}|u_n|_{2}^{2}-e^{(m+(\gamma+1)k)s_n}\displaystyle \int_{\mathbb{R}^N}F(u_n)\,dz=\tilde{d}+o_n(1) \\
\mbox{} \\
\frac{(m+(\gamma+1)-2)e^{(m+(\gamma+1)-2)s_n}}{2}\|u_n\|_{\mathcal{D}_{\gamma}^{1,2}\big(\mathbb{R}^{N}\big)}^{2}+\frac{(m+(\gamma+1)k)e^{(m+(\gamma+1)k)s_n}}{2}|u_n|_{2}^{2}-\\ 
\mbox{}\\
\hspace{4 cm} -(m+(\gamma+1)k)e^{(m+(\gamma+1)k)s_n}\displaystyle\int_{\mathbb{R}^N}F(u_n)\,dz=o_n(1) \\
\mbox{}\\
{e^{(m+(\gamma+1)-2)s_n}}\|u_n\|_{\mathcal{D}_{\gamma}^{1,2}\big(\mathbb{R}^{N}\big)}^{2}+{e^{(m+(\gamma+1)k)s_n}}|u_n|_{2}^{2}-e^{(m+(\gamma+1)k)s_n}\displaystyle \int_{\mathbb{R}^N}f(u_n)u_n\,dz=o_n(1).
\end{array}
\right.
$$

The first and second equation of the previous system ensure that $(\|u_n\|_{\mathcal{D}_{\gamma}^{1,2}\big(\mathbb{R}^{N}\big)})$ is bounded. From $(f_1)-(f_2)$, there is $C>0$ such that
$$
|f(t)t| \leq \frac{1}{4}|t|^{2}+C|t|^{2^*}, \quad \forall t \in \mathbb{R}.
$$
Using this information in the last equation of the above system, we get
$$
|u_n|_{2}^{2} \leq C\int_{\mathbb{R}^N}|u_n|^{2^*}\,dz+o_n(1), \quad \forall n \in \mathbb{N}.
$$
Now, employing (\ref{Sobolev-inequality}), we find
$$
|u_n|_{2}^{2} \leq C\|u_n\|_{\mathcal{D}_{\gamma}^{1,2}\big(\mathbb{R}^{N}\big)}^{2}+o_n(1), \quad \forall n \in \mathbb{N},
$$
from where it follows that $(u_n)$ is a bounded sequence in $\mathcal{H}_{\gamma,rad_{x,y}}^{1,2}(\mathbb{R}^{N}) $.	
\end{proof}

\subsection{Proof of Theorem \ref{Theorem1}}
By Proposition \ref{P1}, there is $\{(s_n,u_n)\} \subset \mathbb{R} \times \mathcal{H}_{\gamma,rad_{x,y}}^{1,2}(\mathbb{R}^{N})$  satisfying
\begin{itemize}
	\item[(a)] $s_n \to 0;$ 	
	\item[(b)] $\tilde{I}(s_n,u_n) \to \tilde{d}; $
	\item[(c)] $\frac{\partial}{\partial u}\tilde{I}(s_n,u_n)\to 0$ strongly in $(\mathcal{H}_{\gamma,rad_{x,y}}^{1,2}(\mathbb{R}^{N}))^*$;
	\item[(d)] $\frac{\partial}{\partial s}\tilde{I}(s_n,u_n)\to 0$ in $\mathbb{R}$.
\end{itemize}	
By Lemma \ref{limitacao}, the sequence $(u_n)$ is bounded in $\mathcal{H}_{\gamma,rad_{x,y}}^{1,2}(\mathbb{R}^{N})$, then for some subsequence, still denoted by itself, there is $u_0 \in \mathcal{H}_{\gamma,rad_{x,y}}^{1,2}(\mathbb{R}^{N})$ such that
$$
u_n \rightharpoonup u_0 \quad \mbox{in} \quad \mathcal{H}_{\gamma,rad_{x,y}}^{1,2}(\mathbb{R}^{N}).
$$
The item $(d)$ implies that $\frac{\partial}{\partial u}\tilde{I}(s_n,u_n)w\to 0$ for all $w \in \mathcal{H}_{\gamma,rad_{x,y}}^{1,2}(\mathbb{R}^{N})$, that is, 
$$
e^{(m+(\gamma+1)-2)s_n}\int_{\mathbb{R}^{N}}\nabla_{\gamma} u_n\nabla_{\gamma} w\,dz+{e^{(m+(\gamma+1)k)s_n}}\int_{\mathbb{R}^N}u_nw\,dz-e^{(m+(\gamma+1)k)s_n}\displaystyle \int_{\mathbb{R}^N}f(u_n)w\,dz=o_n(1).
$$
Since $s_n \to 0$, it is easy to check that the above equality leads to
$$
\int_{\mathbb{R}^{N}}\nabla_{\gamma} u_0\nabla_{\gamma} w\,dz+\int_{\mathbb{R}^N}u_0w\,dz-\displaystyle \int_{\mathbb{R}^N}f(u_0)w\,dz=0, \quad \forall w \in \mathcal{H}_{\gamma,rad_{x,y}}^{1,2}(\mathbb{R}^{N}),
$$
showing that $I'(u_0)w=0$ for all $w \in \mathcal{H}_{\gamma,rad_{x,y}}^{1,2}(\mathbb{R}^{N})$, that is, $u_0$ is a critical point of $I$. We claim that $u_0 \not= 0$, because if $u_0=0$, the Lemma \ref{embedding-compact0} gives that 
\begin{equation} \label{limitOD}
\int_{\mathbb{R}^N}|u_n|^{p}\,dz \to 0, \quad \forall p \in (2,2^{*}_{\gamma}).
\end{equation}
From $(f_1)-(f_2)$, given $\epsilon>0$, there is $C>0$  such that
$$
\int_{\mathbb{R}^N}|f(u_n)u_n| \,dz\leq \epsilon \int_{\mathbb{R}^N}|u_n|^{2}\,dz+C\int_{\mathbb{R}^N}|u_n|^{p}\,dz.
$$ 
As $(u_n)$ is bounded in $\mathcal{H}_{\gamma,rad_{x,y}}^{1,2}\big(\mathbb{R}^{N}\big)$, it follows that
$$
\int_{\mathbb{R}^N}|f(u_n)u_n| \,dz\leq \epsilon M+C\int_{\mathbb{R}^N}|u_n|^{p}\,dz, 
$$ 
where $M=\displaystyle \sup_{n \in \mathbb{N}} \int_{\mathbb{R}^N}|u_n|^{2}\,dz$. Now, using the limit (\ref{limitOD}), we derive that
$$
\limsup_{n \to +\infty} \int_{\mathbb{R}^N}|f(u_n)u_n|\,dz \leq \epsilon M.
$$
As $\epsilon$ is arbitrary, the last limit yields 
$$
\lim_{n \to +\infty} \int_{\mathbb{R}^N}f(u_n)u_n\,dz=0.
$$
This limit combined together with the limit $\frac{\partial}{\partial u}\tilde{I}(s_n,u_n)u_n=o_n(1)$ allows to deduce that $u_n \to 0$ in $\mathcal{H}_{\gamma,rad_{x,y}}^{1,2}\big(\mathbb{R}^{N}\big)$. Hence, 
$$
\tilde{I}(s_n,u_n) \to 0, 
$$
which is absurd, because $\tilde{I}(s_n,u_n) \to \tilde{d}>0$. Thus, $u_0$ is a nontrivial critical point of $I$ in $\mathcal{H}_{\gamma,rad_{x,y}}^{1,2}\big(\mathbb{R}^{N}\big)$.

Next, we will show that $u_0$ is in fact a critical point of $I$ in the space $\mathcal{H}_{\gamma}^{1,2}\big(\mathbb{R}^{N}\big)$. To see why, we will apply the {\it Principle of Symmetric Criticality} \cite{Palais1979MP} ( see also \cite[Theorem 1.28]{Willem1996} ). In what follows, let us denote by $G$ the following subgroup of $O(N)$:
\begin{equation*}
G=\left\{g\in O(N)\;:\;g(x,y)=(h(x),l(y)) ~ \hbox{with} ~ h\in {O}(m) \quad \mbox{and} \quad l \in {O}(k) \right\}
\end{equation*}
and the action $G\times \mathcal{H}_{\gamma}^{1,2}\big(\mathbb{R}^{N}\big) \longrightarrow \mathcal{H}_{\gamma}^{1,2}\big(\mathbb{R}^{N}\big) $ is given by
\begin{equation*}
(gu)(x,y)=u(h(x),l(y)) ~ \hbox{for all } ~ (x,y)\in\mathbb{R}^{m}\times\mathbb{R}^{k},
\end{equation*}
which is isometric, that is, $\|gu\|_{\gamma}=\|u\|_{\gamma}$. Furthermore, the functional $I$ is invariant under the action of $G$, since for all $g\in G$ and $u \in \mathcal{H}_{\gamma}^{1,2}\big(\mathbb{R}^{N}\big)$,
$$
I(gu)  = \frac{1}{2}||gu||^2_{\gamma} -\int_{\mathbb{R}^{N}}F(gu) \,dz=\frac{1}{2}||u||^2_{\gamma} - \int_{\mathbb{R}^{N}}F(u) \,dz = I(u).
$$
From the above commentaries,  it is easy to check that 
$$
	Fix(G)  =  \{ u\in \mathcal{H}_{\gamma}^{1,2}\big(\mathbb{R}^{N}\big)\,:\, gu=u, \forall g\in G\}=\mathcal{H}_{\gamma,rad_{x,y}}^{1,2}\big(\mathbb{R}^{N}\big).
$$
Hence, by the Principle of Symmetric Criticality of Palais, $u_0$ is a nontrivial critical point of $I$ in $\mathcal{H}_{\gamma}^{1,2}(\mathbb{R}^{N})$, and consequently $u_0$ is the required nontrivial solution of $(P)_1$. Finally, we can see that $u_0$ is nonnegative, because it satisfies 
\begin{equation*}
\int_{\mathbb{R}^{N}}\big(\nabla_{\gamma} u_0\nabla_{\gamma}\varphi +u_0\varphi \big)\,dz =\int_{\mathbb{R}^{N}} f(u_0)\varphi \,dz, \quad \varphi \in \mathcal{H}_{\gamma}^{1,2}\big(\mathbb{R}^{N}\big).
\end{equation*}
Then,  choosing the test function $\varphi=u_0^{-}:=\min\{u,0\}$, we find that $||u_0^{-}||_{\gamma}=0$. This proves that $u_0$ is nonnegative.
\\

\section{Proof of Theorem \ref{Theorem1'}}

In what follows, as in the previous section, without loss of generality, we assume that 
$$
f(s)=0, \quad \forall s \leq 0.
$$

The energy functional associated with problem $(P)_0$ is the functional  $J:\mathcal{D}_{\gamma}^{1,2}\big(\mathbb{R}^{N}\big)\longrightarrow\mathbb{R}$ given by 
\begin{equation*}
J(u)=\frac{1}{2}\int_{\mathbb{R}^{N}}|\nabla_{\gamma}u|^2  \, dz - \int_{\mathbb{R}^{N}} F(u)\, dz.
\end{equation*}
The reader is invited to see that weak solutions of problem $(P)_0$ are critical points of $I$. 

In order to overcome the loss of compactness involving the space $\mathcal{D}_{\gamma}^{1,2}\big(\mathbb{R}^{N}\big)$, we will restrict to functional $J$ to the space $\mathcal{D}_{\gamma,rad_{x,y}}^{1,2}\big(\mathbb{R}^{N}\big)$. In this section, $\|\,\,\|$ denotes the norm $\|\,\,\,\|_{\mathcal{D}_{\gamma}^{1,2}\big(\mathbb{R}^{N}\big)}$.

Our first lemma shows that $J$ verifies the mountain pass geometry. However, we will omit its proof, because the arguments are similar those used in the proof of Lemma \ref{MP-geometry}. 
\begin{lemma}\label{MP-geometryD} There exist positive numbers $\rho$ and $\alpha$ such that:
	\begin{description}
		\item[i)] $J(u)\geq\alpha$ if  $||u||=\rho$.
		\item[ii)] There exists $v\in \mathcal{D}_{\gamma,rad_{x,y}}^{1,2}\big(\mathbb{R}^{N}\big)$ such that $\|v\|>\rho$ and $J(v)< 0$.
	\end{description}
\end{lemma}

The last lemma permits to conclude that the mountain pass level given by 
\begin{equation*}
d_0:=\inf_{g\in\Gamma}\max_{t\in [0,1]}J(g(t))\geq\alpha>0,
\end{equation*}
where 
\begin{equation*}
\Gamma=\{g\in C([0,1],\mathcal{D}_{\gamma,rad_{x,y}}^{1,2}\big(\mathbb{R}^{N}\big):g(0)=0 ~ \hbox{and} ~ g(1)=v \}
\end{equation*}
is well defined.

Arguing as in  Section 3, we will consider an auxiliary functional $\tilde{J}:\mathbb{R} \times \mathcal{D}_{\gamma,rad_{x,y}}^{1,2}\big(\mathbb{R}^{N}\big) \to \mathbb{R}$ given by 
$$
\tilde{J}(s,u(z))=J(u(e^{-s}x,e^{-(\gamma+1)s}y)), \quad z=(x,y) \in \mathbb{R}^N=\mathbb{R}^m \times \mathbb{R}^k,
$$
that is,
$$
\tilde{J}(s,u)= \frac{e^{(m+(\gamma+1)-2)s}}{2}\|u\|^{2}-e^{(m+(\gamma+1)k)s}\int_{\mathbb{R}^N}F(u)\,dz.
$$

Is is easy to prove that $\tilde{J}$ also satisfies the mountain pass geometry, more precisely, we have the result below 

\begin{lemma}\label{MP-geometry2D} There exist positive numbers $\rho$ and $\alpha$ such that:
	\begin{description}
		\item[i)] $\tilde{J}(s,u)\geq\alpha$ if  $\|(s,u)\|_{\mathbb{R} \times \mathcal{D}_{\gamma}^{1,2}\big(\mathbb{R}^{N}\big)}=\rho$.
		\item[ii)] There exists $v\in \mathcal{D}_{\gamma}^{1,2}\big(\mathbb{R}^{N}\big)$ such that $\|(0,v)\|_{\mathbb{R} \times \mathcal{D}_{\gamma}^{1,2}\big(\mathbb{R}^{N}\big)}=\|v\|>\rho$ and $\tilde{J}(0,v)=J(v)< 0$.
	\end{description}
\end{lemma}

From now on, we denote by $\tilde{d}_0$, the mountain pass level associated with $\tilde{J}$ given by 
\begin{equation*}
\tilde{d}_0:=\inf_{\tilde{g}\in\Gamma}\max_{t\in [0,1]}\tilde{J}(\tilde{g}(t))\geq\alpha>0,
\end{equation*}
where 
\begin{equation*}
\Gamma=\{\tilde{g} \in C([0,1],\mathbb{R} \times \mathcal{D}_{\gamma,rad_{x,y}}^{1,2}\big(\mathbb{R}^{N}\big):\tilde{g}(0)=0 ~ \hbox{and} ~ \tilde{g}(1)=(0,v) \}.
\end{equation*}

From definition of $J$ and $\tilde{J}$, we also derive that equality $d_0=\tilde{d}_0$, which permits to repeat the same arguments explored in the proof of \cite[Propostion 4.2]{JNK} ( see \cite[Propostion 2.2]{Jeanjean}) to get the following proposition

\begin{prop} \label{P1D} There is a sequence $\{(s_n,u_n)\} \subset \mathbb{R} \times \mathcal{D}_{\gamma,rad_{x,y}}^{1,2}\big(\mathbb{R}^{N}\big)$ such that
	\begin{itemize}
		\item[(a)] $s_n \to 0;$ 	
		\item[(b)] $\tilde{J}(s_n,u_n) \to \tilde{d}; $
		\item[(c)] $\frac{\partial}{\partial u}\tilde{J}(s_n,u_n)\to 0$ strongly in $(\mathcal{D}_{\gamma,rad_{x,y}}^{1,2}\big(\mathbb{R}^{N}\big))^*$;
		\item[(d)] $\frac{\partial}{\partial s}\tilde{J}(s_n,u_n)\to 0$ in $\mathbb{R}$.
	\end{itemize}		
	
\end{prop}	

The next lemma establishes the boundedness of the sequence $(u_n) \subset \mathcal{D}_{\gamma,rad_{x,y}}^{1,2}\big(\mathbb{R}^{N}\big) $ that was obtained in the last proposition. Since the proof follows the same spirit of Lemma \ref{limitacao}, we also omit its proof.

\begin{lemma} \label{limitacaoD} The sequence $(u_n) \subset \mathcal{D}_{\gamma,rad_{x,y}}^{1,2}\big(\mathbb{R}^{N}\big) $ is bounded. 
	
\end{lemma}	 

\subsection{Proof of Theorem \ref{Theorem1'}}
By Proposition \ref{P1D}, there is $\{(s_n,u_n)\} \subset \mathbb{R} \times \mathcal{D}_{\gamma,rad_{x,y}}^{1,2}\big(\mathbb{R}^{N}\big)$  satisfying
\begin{itemize}
	\item[(a)] $s_n \to 0;$ 	
	\item[(b)] $\tilde{J}(s_n,u_n) \to \tilde{d}; $
	\item[(c)] $\frac{\partial}{\partial u}\tilde{J}(s_n,u_n)\to 0$ strongly in $(\mathcal{D}_{\gamma,rad_{x,y}}^{1,2}\big(\mathbb{R}^{N}\big))^*$;
	\item[(d)] $\frac{\partial}{\partial s}\tilde{J}(s_n,u_n)\to 0$ in $\mathbb{R}$.
\end{itemize}	
By Lemma \ref{limitacao}, the sequence $(u_n)$ is bounded in $\mathcal{D}_{\gamma,rad_{x,y}}^{1,2}\big(\mathbb{R}^{N}\big)$, then for some subsequence, still denoted by itself, there is $u_0 \in \mathcal{D}_{\gamma,rad_{x,y}}^{1,2}\big(\mathbb{R}^{N}\big)$ such that
$$
u_n \rightharpoonup u_0 \quad \mbox{in} \quad \mathcal{D}_{\gamma,rad_{x,y}}^{1,2}\big(\mathbb{R}^{N}\big).
$$
The item $(d)$ implies that $\frac{\partial}{\partial u}\tilde{J}(s_n,u_n)w\to 0$ for all $w \in \mathcal{D}_{\gamma,rad_{x,y}}^{1,2}\big(\mathbb{R}^{N}\big)$, that is, 
$$
e^{(m+(\gamma+1)-2)s_n}\int_{\mathbb{R}^{N}}\nabla_{\gamma} u_n\nabla_{\gamma} w\,dz-e^{(m+(\gamma+1)k)s_n}\displaystyle \int_{\mathbb{R}^N}f(u_n)w\,dz=o_n(1).
$$
Since $s_n \to 0$, a simple computation guarantees that $u_0$ satisfies the equality below 
$$
\int_{\mathbb{R}^{N}}\nabla_{\gamma} u_0\nabla_{\gamma} w\,dz-\displaystyle \int_{\mathbb{R}^N}f(u_0)w\,dz=0,
$$
showing that $J'(u_0)w=0$ for all $w \in \mathcal{D}_{\gamma,rad_{x,y}}^{1,2}\big(\mathbb{R}^{N}\big)$, that is, $u_0$ is a critical point of $J$ in $\mathcal{D}_{\gamma,rad_{x,y}}^{1,2}\big(\mathbb{R}^{N}\big)$. We claim that $u_0 \not= 0$, because if $u_0=0$, the Lemma \ref{embedding-compact01} gives that 
\begin{equation} \label{limitD}
\int_{\{a<|u_n|<b\}}|u_n|^{p}\,dz \to 0,
\end{equation}
for all $0<a<b<+\infty$ and $p \in (2,2^{*}_{\gamma})$.

From $(f_4)-(f_5)$, given $\epsilon>0$ and $p \in (2,2^{*}_{\gamma})$, there are  $C>0$ and $a<b<+\infty$ such that
$$
\int_{\mathbb{R}^N}|f(u_n)u_n| \,dz\leq \epsilon \int_{\mathbb{R}^N}|u_n|^{2^{*}_{\gamma}}\,dz+C\int_{\{a<|u_n|<b\}}|u_n|^{p}\,dz.
$$ 
As $(u_n)$ is bounded in $\mathcal{D}_{\gamma,rad_{x,y}}^{1,2}\big(\mathbb{R}^{N}\big)$, it follows that
$$
\int_{\mathbb{R}^N}|f(u_n)u_n| \,dz\leq \epsilon M+C\int_{\{a<|u_n|<b\}}|u_n|^{p}\,dz, 
$$ 
where $M=\displaystyle \sup_{n \in \mathbb{N}}\int_{\mathbb{R}^N}|u_n|^{2^{*}_{\gamma}}\,dz$. Now, using the limit (\ref{limitD}), we see that
$$
\limsup_{n \to +\infty} \int_{\mathbb{R}^N}|f(u_n)u_n|\,dz \leq \epsilon M.
$$
Thereby, as $\epsilon$ is arbitrary, we can infer that 
$$
\lim_{n \to +\infty} \int_{\mathbb{R}^N}f(u_n)u_n\,dz=0.
$$
This limit combined with the limit $\frac{\partial}{\partial u}\tilde{J}(s_n,u_n)u_n=o_n(1)$ ensures that $u_n \to 0$ in $\mathcal{D}_{\gamma,rad_{x,y}}^{1,2}\big(\mathbb{R}^{N}\big)$. Hence, 
$$
\tilde{J}(s_n,u_n) \to 0, 
$$
which is absurd, because $\tilde{J}(s_n,u_n) \to \tilde{d}_0>0$. Therefore, $u_0$ is a nontrivial critical point of $J$ in $\mathcal{D}_{\gamma,rad_{x,y}}^{1,2}\big(\mathbb{R}^{N}\big)$. As in Section 3, we can again use the {\it Principle of Symmetric Criticality} to prove that $u_0$ is a critical point of $J$ in  $\mathcal{D}_{\gamma}^{1,2}\big(\mathbb{R}^{N}\big)$. 
\section{Proof of Theorem \ref{Theorem2}}

In this section, the energy functional associated with $(P)_a$ is given by 
\begin{equation*}
I(u)=\frac{1}{2}\int_{\mathbb{R}^{N}}\big(|\nabla_{\gamma}u|^2 +a(z)|u|^2\big) \, dz - \int_{\mathbb{R}^{N}} F(u)\, dz, \quad \forall u \in \mathcal{H}_{\gamma}^{a}(\mathbb{R}^N).
\end{equation*}

 Unfortunately the space $\mathcal{H}_{\gamma}^{a}(\mathbb{R}^N)$ does not have good compactness embedding to help us to find a nontrivial critical point $u$ for $I$. Following the same approach explored in the previous sections, we will restrict $I$ to the space $\mathcal{H}^{a}_{\gamma,rad_{x}}(\mathbb{R}^N)$. Since $I$  satisfies the mountain pass geometry in $\mathcal{H}^{a}_{\gamma,rad_{x}}(\mathbb{R}^N)$, there is $(u_n) \subset \mathcal{H}^{a}_{\gamma,rad_{x}}(\mathbb{R}^N)$ such that 
\begin{equation} \label{MP}
	I(u_n)\longrightarrow d_a ~~ \hbox{and} ~~ I'(u_n)\longrightarrow 0,
\end{equation}
where 
\begin{equation*}
d_a:=\inf_{g\in\Gamma}\max_{t\in [0,1]}I(g(t))\geq\alpha>0,
\end{equation*}
where 
\begin{equation*}
\Gamma_a=\{g\in C([0,1], \mathcal{H}^{a}_{\gamma,rad_{x}}(\mathbb{R}^N)\,:\,g(0)=0 ~ \hbox{and} ~ g(1)=v \}.
\end{equation*}

Using well known arguments, the condition $(f_7)$ helps us to prove that $(u_n)$ is bounded in $\mathcal{H}_{\gamma}^{a}(\mathbb{R}^N)$. Since $\mathcal{H}^{a}_{\gamma,rad_{x}}(\mathbb{R}^N)$ is reflexive, we can assume that there exists $u_0\in \mathcal{H}^{a}_{\gamma,rad_{x}}(\mathbb{R}^N)$ such that, up to subsequence,
\begin{description}
	\item[] $ \bullet ~~ u_n \rightharpoonup u_0 ~~ in ~~ \mathcal{H}^{a}_{\gamma,rad_{x}}(\mathbb{R}^N), $
	\item[] $ \bullet ~~ u_n \rightarrow u_0 ~~ in ~~ L_{loc}^{p}(\mathbb{R}^N) ~~ \hbox{for all} ~~ p\in(2,2^{*}_{\gamma})$,
	\item[] $ \bullet ~~ u_n \rightarrow u_0 ~~ a. e. ~~ in ~~ \mathbb{R}^N $.
\end{description}
As $I'(u_n)v = o_n(1)$ for all  $v \in \mathcal{H}^{a}_{\gamma,rad_{x}}(\mathbb{R}^N)$, the above convergences yield
$$
I'(u_0)v = 0 ~~ \hbox{for all} ~~ v \in\mathcal{H}^{a}_{\gamma,rad_{x}}(\mathbb{R}^N),
$$
from where it follows that $u_0$ is a critical point of $I$.

\begin{lemma}\label{Lions-lemma-consequence}
	
	Let $(u_n)\subset \mathcal{H}^{a}_{\gamma,rad_{x}}(\mathbb{R}^N)$ be a sequence such that
	$$
	I(u_n)\longrightarrow d_a>0 ~~ \hbox{and} ~~ I'(u_n)\longrightarrow 0.
	$$
	Then, there exist $(y_n)\subset\mathbb{R}^{k}$ and $r, \beta>0$ such that
	\begin{equation*}
	\int_{\mathbb{R}^{m}\times B_r(y_n)}|u_n|^{q}dz\geq\beta, ~~ \forall n\in\mathbb{N} .
	\end{equation*}	
\end{lemma}

\begin{proof}
In what follows we will argue by contradiction. If the lemma is not true, we can employ Lemma \ref{Lions-lemma} to get
	$$
	u_n\to 0 ~~ \hbox{in} ~~ L^{q}(\mathbb{R}^{N}) ~~ \hbox{for} ~~ 2<q<2^{*}_{\gamma}.
	$$
	From $(f_1)$ and $(f_2)$, for each $\varepsilon>0$, there is a constant $C_{\varepsilon}>0$ such that
	\begin{equation*}
	|f(u_n)u_n|\leq \varepsilon |u_n|^2+C_{\varepsilon}|u_n|^{q}, ~~ \hbox{for all} ~~ n\in\mathbb{N}.
	\end{equation*}
	Hence,
	$$
	\int_{\mathbb{R}^{N}}|f(u_n)u_n|\, dz\leq\varepsilon C+C_{\varepsilon}\int_{\mathbb{R}^{N}}|u_n|^{q}\, dz,
	$$
and so, 
	$$
	\limsup_{n}\int_{\mathbb{R}^{N}}|f(u_n)u_n|\, dz\leq\varepsilon C,
	$$
leading to
	$$
	\lim_{n}\int_{\mathbb{R}^{N}}|f(u_n)u_n|\, dz=0.
	$$
As $I'(u_n)u_n = o_n(1)$, we have
	$$
	||u_n||_{\gamma}^{2}=I'(u_n)u_n+\int_{\mathbb{R}^{N}}|f(u_n)u_n|\, dz=o_n(1),
	$$
	that is, 
	$$
	u_n\to 0 ~~ \hbox{in} ~~ \mathcal{H}^{a}_{\gamma,rad_{x}}(\mathbb{R}^N).
	$$
	Therefore
	$$
	I(u_n)\to I(0)=0,
	$$
	which is a contradiction, because 	$I(u_n)\longrightarrow d_a>0$.
\end{proof}

\subsection{Proof of Theorem \ref{Theorem2}}

In what follows, we will divide the proof into two cases:  \\

\noindent{\bf Case 1: \eqref{A3} holds.}

In this case, we can assume that the weak limit $u_0 \in \mathcal{H}^{a}_{\gamma,rad_{x}}(\mathbb{R}^N)$ of the sequence given in \eqref{MP} is nontrivial.  Otherwise, by Lemma \ref{Lions-lemma-consequence}, there exist $(y_n)\subset\mathbb{R}^{k}$ and $r, \beta>0$ such that
\begin{equation*}
\int_{\mathbb{R}^{m}\times B_r(y_n)}|u_n|^{2} \, dz \geq\beta, ~~ \forall n\in\mathbb{N}.
\end{equation*}
Increasing $r$ if necessary, we can suppose that $y_n \in \mathbb{Z}^k$. Now, setting the sequence  
$$
v_n(z)=v_n(x,y)=u_n(x,y+y_n), ~~ \forall z=(x,y)\in\mathbb{R}^{N},
$$
it follows that $(v_n)\subset\mathcal{H}_{\gamma, rad_x}^{a}(\mathbb{R}^{N})$. As we are supposing that \eqref{A3} occurs, a simple computation shows that
$$
||v_n||_{a}=||u_n||_{a}.
$$
Hence $(v_n)$ is bounded,  and for a subsequence of $(v_n)$, still denoted by itself, there is $v\in \mathcal{H}_{\gamma, rad_x}^{a}(\mathbb{R}^{N})$ such that $v_n \rightharpoonup v$ in $\mathcal{H}_{\gamma, rad_x}^{a}(\mathbb{R}^{N})$. By Lemma \ref{embedding-compact} $v_n \rightarrow v$ in $L^{p}(\mathbb{R}^{m}\times B_r(0))$ for all $p\in(2,2^{*}_{\gamma})$ and $v_n(z) \rightarrow v(z)$ almost everywhere in $\mathbb{R}^N $. Furthermore,
$$
I(v_n)=I(u_n) ~~ \hbox{and} ~~ I'(v_n)\longrightarrow 0,
$$
from where it follows that $I'(v)w=0$ for all $w\in \mathcal{H}_{\gamma, rad_x}^{a}(\mathbb{R}^{N})$, showing that $v$ is a critical point for $I$ in $\mathcal{H}_{\gamma, rad_x}^{a}(\mathbb{R}^{N})$. On the other hand, the equality 
$$
\int_{\mathbb{R}^{m}\times B_r(0)}|v_n|^{q} \, dz=\int_{\mathbb{R}^{m}\times B_r(y_n)}|u_n|^{q} \, dz\geq\beta, ~~ \forall n\in\mathbb{N},
$$
together with the limit $v_n \rightarrow v$ in $L^{p}(\mathbb{R}^{m}\times B_r(0))$ (see Lemma \ref{embedding-compact}) gives 
\begin{equation*}
\int_{\mathbb{R}^{m}\times B_r(0)}|v|^{q} \, dz\geq\beta, ~~ \forall n\in\mathbb{N}, 
\end{equation*}
showing that $v\neq 0$. \\

\noindent{\bf Case 1: \eqref{A4} holds.} 

If \eqref{A4} holds,  the Lemma \ref{embedding-compact-coercive} implies that 
$$
u_n \rightarrow u_0 ~~ in ~~ L^{p}(\mathbb{R}^N) ~~ \hbox{for all} ~~ p\in[2,2^{*}_{\gamma}).
$$
This limit combined with Lemma \ref{embedding-compact-coercive} ensures that 
$$
\int_{\mathbb{R}^{N}}f(u_n)(u_n-u_0) \,dz\longrightarrow 0.
$$
Now, since $I'(u_n)\to 0$ and $u_n \rightharpoonup u_0$ in $\mathcal{H}_{\gamma, rad_x}^{a}(\mathbb{R}^{N})$, we find that
$$
||u_n-u_0||_{a}^{2}  =  I'(u_n)u_n + \int_{\mathbb{R}^{N}}f(u_n)u_n \,dz  - I'(u_n)u_n - \int_{\mathbb{R}^{N}}f(u_n)u_0 \,dz + o_n(1)=o_n(1).
$$
which shows that $u_n \to u_0$ in $\mathcal{H}_{\gamma, rad_x}^{a}(\mathbb{R}^{N})$. From this, 
$$
I(u_0)=d_a>0,
$$
and so, $u_0$ is a nontrivial critical point for $I$ in $\mathcal{H}_{\gamma, rad_x}^{a}(\mathbb{R}^{N})$. 

Finally, in order to conclude the proof, we must show that $u$ is a critical point for $I$ in $\mathcal{H}_{\gamma}^{a}(\mathbb{R}^{N})$. As in the  Sections 3 and 4, it is enough to apply the  {\it Principle of Symmetric Criticality}. Since the argument are the same we omit its proof.
\section{Proof of Theorem \ref{Theorem3}}
Suppose that $u\in\mathcal{H}^{a}(\mathbb{R}^N)$ is a solution of $(P)_a$. For any $R>0$, $0<r\leq R/2$, let $\eta\in C^{\infty}(\mathbb{R}^{N})$, $0\leq\eta\leq 1$ with
\begin{equation*}
\eta(z)=\left \{
\begin{array}{ccl}
1 & \mbox{if} & |z|\geq R, \\
0 & \mbox{if} &  |z|\leq R-r, \\
\end{array}
~~ \hbox{and} ~~ |\nabla\eta|\leq\frac{2}{r\sqrt{R_{\gamma}}}
\right.
\end{equation*}
where $R_{\gamma}=\max\{1, R^{2\gamma}\}$. For each $L>0$, let us define
\begin{equation*}
u_{L}=\left \{
\begin{array}{ccl}
u & \mbox{if}  & u\leq L, \\
L & \mbox{if} &  u\geq L, \\
\end{array}
\right.
\end{equation*}
$v_L=\eta^2uu_{L}^{2(\beta-1)}$ and $w_L=\eta uu_{L}^{\beta-1}$ with $\beta>1$ to be determined later on. Then 
\begin{equation*}
\nabla_{\gamma}u\nabla_{\gamma}v_{L}=\eta^{2}u_{L}^{2(\beta-1)}|\nabla_{\gamma}u|^2+2(\beta-1)\eta^{2}uu_{L}^{2\beta-3}\nabla_{\gamma}u\nabla_{\gamma}u_{L}+2\eta uu_{L}^{2(\beta-1)}\nabla_{\gamma} u\nabla_{\gamma}\eta
\end{equation*}
and choosing $v_L$ as test function, we obtain
\begin{equation*}
\begin{split}
\int_{\mathbb{R}^{N}}\eta^{2}u_{L}^{2(\beta-1)}|\nabla_{\gamma}u|^2\,dz&+2(\beta-1)\int_{\mathbb{R}^{N}}\eta^{2}uu_{L}^{2\beta-3}\nabla_{\gamma}u\nabla_{\gamma}u_{L}\,dz +\\ &\int_{\mathbb{R}^{N}}a(z)\eta^{2}u^2u_{L}^{2(\beta-1)} \, dz= \int_{\mathbb{R}^{N}} f(u)\eta^{2}uu_{L}^{2(\beta-1)} \,dz-2\int_{\mathbb{R}^{N}}\eta uu_{L}^{2(\beta-1)}\nabla_{\gamma} u\nabla_{\gamma}\eta \, dz.
\end{split}
\end{equation*}
As
\begin{equation*}
\int_{\mathbb{R}^{N}}\eta^2uu_{L}^{2\beta-3}\nabla_{\gamma}u\nabla_{\gamma}u_{L}\,dz=\int_{[u\leq L]}\eta^2u^{2(\beta-1)}|\nabla_{\gamma}u|^2 \,dz\geq 0
\end{equation*}
and $a(z)\geq a_0>0$ in $\mathbb{R}^{N}$, we derive that 
\begin{equation*}
\int_{\mathbb{R}^{N}}\eta^2u_{L}^{2(\beta-1)}|\nabla_{\gamma}u|^2\,dz+a_0\int_{\mathbb{R}^{N}}\eta^{2}u^2u_{L}^{2(\beta-1)} \, dz\leq \int_{\mathbb{R}^{N}} f(u)\eta^2 uu_{L}^{2(\beta-1)} \,dz-2\int_{\mathbb{R}^{N}}\eta uu_{L}^{2(\beta-1)}\nabla_{\gamma} u\nabla_{\gamma}\eta \, dz.
\end{equation*}
By condition $(f_1)$ and $(f_2)$, $f(u)\leq\frac{a_0}{2}u+Cu^{q-1}$, and so, 
\begin{equation*}
\int_{\mathbb{R}^{N}}\eta^2u_{L}^{2(\beta-1)}|\nabla_{\gamma}u|^2\,dz +\frac{a_0}{2}\int_{\mathbb{R}^{N}}\eta^{2}u^2u_{L}^{2(\beta-1)} \, dz \leq C\int_{\mathbb{R}^{N}} \eta^2u^qu_{L}^{2(\beta-1)} \,dz +2\int_{\mathbb{R}^{N}}\eta uu_{L}^{2(\beta-1)}|\nabla_{\gamma} u||\nabla_{\gamma}\eta| \, dz.
\end{equation*}
Using Young's inequality with $\varepsilon>0$ sufficiently small, we get  
\begin{equation}\label{By-Young}
\begin{split}
\int_{\mathbb{R}^{N}}\eta^2u_{L}^{2(\beta-1)}|\nabla_{\gamma}u|^2\,dz +\frac{a_0}{2}\int_{\mathbb{R}^{N}}\eta^{2}u^2u_{L}^{2(\beta-1)} \, dz & \leq C\int_{\mathbb{R}^{N}} \eta^2u^qu_{L}^{2(\beta-1)} \,dz \\ &+C_{\varepsilon}\int_{\mathbb{R}^{N}}u^2u_{L}^{2(\beta-1)}|\nabla_{\gamma}\eta|^2 \, dz.
\end{split}
\end{equation}
From \eqref{embedding-continuous}, 
\begin{equation*}
|w_L|_{2^{*}_{\gamma}}^{2}\leq\mathcal{C}\left(\int_{\mathbb{R}^{N}}|\nabla_{\gamma}w_L|^2\,dz+\int_{\mathbb{R}^{N}}|w_L|^2\,dz\right),
\end{equation*}
which leads to 
\begin{equation*}
\begin{split}
|w_L|_{2^{*}_{\gamma}}^{2}\leq\mathcal{C}&\int_{\mathbb{R}^{N}}\left(\eta^2u_{L}^{2(\beta-1)}|\nabla_{\gamma}u|^2+ \eta^2(\beta-1)^2u^2u_{L}^{2(\beta-2)}|\nabla_{\gamma}u_{L}|^2+u^2u_{L}^{2(\beta-1)}|\nabla_{\gamma}\eta|^2\right) \, dz \\
&+\mathcal{C}\int_{\mathbb{R}^{N}}\eta^2u^2u_{L}^{2(\beta-1)}\, dz, 
\end{split}
\end{equation*}
and thus
\begin{equation*}
|w_L|_{2^{*}_{\gamma}}^{2}\leq\mathcal{C}\beta^2\left(\int_{\mathbb{R}^{N}}\eta^2u_{L}^{2(\beta-1)}|\nabla_{\gamma}u|^2\,dz
+\frac{a_0}{2}\int_{\mathbb{R}^{N}}\eta^2u^2u_{L}^{2(\beta-1)}\,dz + \int_{\mathbb{R}^{N}}u^2u_{L}^{2(\beta-1)}|\nabla_{\gamma}\eta|^2 \, dz\right).
\end{equation*}
Then from \eqref{By-Young}
\begin{equation}\label{By-Young-Next}
|w_L|_{2^{*}_{\gamma}}^{2}\leq\mathcal{C}\beta^2\left(\int_{\mathbb{R}^{N}} (\eta uu_{L}^{\beta-1})^2u^{q-2} \,dz +\int_{\mathbb{R}^{N}}u^2u_{L}^{2(\beta-1)}|\nabla_{\gamma}\eta|^2 \,dz\right).
\end{equation}
Using H\"older inequality with exponents $\frac{{2^{*}_{\gamma}}}{q-2}$ and $\frac{\alpha^{*}}{2}:=\frac{{2^{*}_{\gamma}}}{2^{*}_{\gamma}-(q-2)}$,
\begin{equation*}
|w_L|_{2^{*}_{\gamma}}^{2}\leq\mathcal{C}\beta^2\left(\int_{\mathbb{R}^{N}} (\eta uu_{L}^{\beta-1})^{\alpha^{*}} \,dz\right)^{2/\alpha^{*}}\left(\int_{\mathbb{R}^{N}} u^{2^{*}_{\gamma}} \, dz \right)^{q-2/2^{*}_{\gamma}} +\mathcal{C}\beta^2\int_{\mathbb{R}^{N}}u^2u_{L}^{2(\beta-1)}|\nabla_{\gamma}\eta|^2 \,dz.
\end{equation*}
Since $u\in L^{2^{*}_{\gamma}}(\mathbb{R}^{N})$, we choose $\beta=\frac{{2^{*}_{\gamma}}}{\alpha^{*}}>1$ so that $u\in L^{\alpha^{*}\beta}(\mathbb{R}^{N})\cap L^{2\beta}(\mathbb{R}^{N})$, and so,
\begin{equation*}
|w_L|_{2^{*}_{\gamma}}^{2}\leq\mathcal{C}\beta^2\left(\int_{\mathbb{R}^{N}} u^{\alpha^{*}\beta} \,dz\right)^{2/\alpha^{*}} +\frac{\mathcal{C}\beta^2}{r^2}\int_{\mathbb{R}^{N}}u^{2\beta}\,dz < \infty.
\end{equation*}
Hence
\begin{equation*}
\left(\int_{[|z|\geq R]} (uu_{L}^{\beta-1})^{2^{*}_{\gamma}} \,dz\right)^{2/2_{\gamma}^{*}}\leq K <\infty.
\end{equation*}
By employing the Fatou's lemma in the variable $L$ we conclude that
\begin{equation}\label{By-Fatou}
\int_{[|z|\geq R]} u^{\frac{(2^{*}_{\gamma})^2}{\alpha^{*}}} \,dz <\infty.
\end{equation}
Once again from inequality \eqref{By-Young-Next}, 
\begin{equation*}
|w_L|_{2^{*}_{\gamma}}^{2}\leq\mathcal{C}\beta^2\left(\int_{[|z|\geq R-r]} u^{q-2}u^{2\beta} \, dz +
\int_{[R\geq |z|\geq R-r]}u^{2\beta} \, dz\right).
\end{equation*}
We now consider $t=\frac{(2^{*}_{\gamma})^2}{\alpha^{*}(q-2)}>1$ and $\beta=\frac{2^{*}_{\gamma}(t-1)}{2t}>1$. Then by H\"older inequality with exponents $t/(t-1)$ and $t$,
\begin{eqnarray*}
|w_L|_{2^{*}_{\gamma}}^{2}& \leq & \mathcal{C}\beta^2\biggl[\biggl(\int_{[|z|\geq R-r]} u^{(q-2)t} \, dz\biggr)^{1/t}\biggl(\int_{[|z|\geq R-r]} u^{2\beta t/(t-1)} \, dz\biggr)^{(t-1)/t} \\ 
& & +\biggl(\int_{[R\geq |z|\geq R-r]}1 \, dz\biggr)^{1/t} \biggl(\int_{[R\geq |z|\geq R-r]}u^{2\beta t/(t-1)} \, dz\biggr)^{(t-1)/t}\biggr].
\end{eqnarray*}
It follows from \eqref{By-Fatou} that
\begin{equation}\label{By-Fatou-Next}
|w_L|_{2^{*}_{\gamma}}^{2}\leq\mathcal{C}\beta^2\left(\int_{[|z|\geq R-r]} u^{2\beta t/(t-1)} \, dz\right)^{(t-1)/t}.
\end{equation}
On the other hand, 
\begin{equation*}
\biggl(\int_{[|z|\geq R]} u_{L}^{2^{*}_{\gamma}\beta} \, dz\biggr)^{2/2^{*}_{\gamma}}\leq\left(\int_{\mathbb{R}^{N}} (\eta uu_{L}^{\beta-1})^{2_{\gamma}^{*}} \,dz\right)^{2/2_{\gamma}^{*}}=|w_{L}|_{2^{*}_{\gamma}}^2.
\end{equation*}
This combined with \eqref{By-Fatou-Next} leads to
\begin{equation*}
|u_{L}|^{2\beta}_{L^{2^{*}_{\gamma}\beta}([|z|\geq R])}\leq\mathcal{C}\beta^2\left(\int_{[|z|\geq R-r]} u^{2\beta t/(t-1)} \, dz\right)^{(t-1)/t}=\mathcal{C}\beta^2|u|^{2\beta}_{L^{2^{*}_{\gamma}}([|z|\geq R-r])}.
\end{equation*}
Hence, by Fatou's lemma 
\begin{equation*}
|u|^{2\beta}_{L^{2^{*}_{\gamma}\beta}([|z|\geq R])}\leq\mathcal{C}\beta^2|u|^{2\beta}_{L^{2^{*}_{\gamma}}([|z|\geq R-r])}.
\end{equation*}
Setting  $\chi:=\frac{2^{*}_{\gamma}(t-1)}{2t}$ and $s:=\frac{2t}{t-1}$, the last inequality allows to conclude that
\begin{equation*}
|u|_{L^{\chi^{m+1}s}([|z|\geq R])}\leq\mathcal{C}^{\sum_{i=1}^{m}\chi^{-i}}\chi^{\sum_{i=1}^{m}i\chi^{-i}}\leq|u|_{L^{2^{*}_{\gamma}}([|z|\geq R-r])}.
\end{equation*}
Letting $m\to\infty$ we obtain
\begin{equation*}
|u|_{L^{\infty}([|z|\geq R])}\leq|u|_{L^{2^{*}_{\gamma}}([|z|\geq R-r])},
\end{equation*}
showing that $\displaystyle \lim_{|z|\to\infty}u(z)=0$.

Finally, in order to show that $u\in L^{\infty}(\mathbb{R}^{N})$, we need only show that the above inequality is true on balls of $\mathbb{R}^{N}$, that is, there exist $R>0$ and $\mathcal{C}>0$ such that
\begin{equation*}
|u|_{L^{\infty}(B_{R}(z))}\leq|u|_{L^{2^{*}_{\gamma}}(B_{2R}(z))}, ~ \forall ~ z\in\mathbb{R}^{N}. 
\end{equation*}
The argument follows with some adjusts from the ideas found in Alves and Germano \cite[ Lemma 2.13]{Alves&Germano2020PA}.

\noindent {\it Claudianor O. Alves and Angelo R. F. de Holanda} \\
	Universidade Federal de Campina Grande, \\
	Unidade Acad\^emica de Matem\'atica, \\	
	58429-970, Campina Grande - PB - Brazil. \\
	e-mail: coalves@mat.ufcg.edu.br,\, angelo@mat.ufcg.edu.br
\end{document}